\documentclass[smallextended]{svjour3}       
\smartqed  
\usepackage{amsfonts,amssymb,graphicx}
\usepackage{mathtools}

\usepackage{float}
\makeatletter
\newcommand*\rel@kern[1]{\kern#1\dimexpr\macc@kerna}
\newcommand*\widebar[1]{%
  \begingroup
  \def\mathaccent##1##2{%
    \rel@kern{0.8}%
    \overline{\rel@kern{-0.8}\macc@nucleus\rel@kern{0.2}}%
    \rel@kern{-0.2}%
  }%
  \macc@depth\@ne
  \let\math@bgroup\@empty \let\math@egroup\macc@set@skewchar
  \mathsurround\z@ \frozen@everymath{\mathgroup\macc@group\relax}%
  \macc@set@skewchar\relax
  \let\mathaccentV\macc@nested@a
  \macc@nested@a\relax111{#1}%
  \endgroup
}
\makeatother

\usepackage{color}
\usepackage{algorithmic,algorithm}
\usepackage[ruled,algosection,algo2e,resetcount]{algorithm2e}

\usepackage{caption}
\usepackage{pgfplots}			
\pgfplotsset{compat=1.3}		
\usepackage{epstopdf} 

 \usepackage{multirow}
\usepackage{scalerel,stackengine}
\stackMath
\newcommand\reallywidehat[1]{%
\savestack{\tmpbox}{\stretchto{%
  \scaleto{%
    \scalerel*[\widthof{\ensuremath{#1}}]{\kern-.6pt\bigwedge\kern-.6pt}%
    {\rule[-\textheight/2]{1ex}{\textheight}}
  }{\textheight}%
}{0.5ex}}%
\stackon[1pt]{#1}{\tmpbox}%
}

\usepackage{geometry}
\newcommand{\RR}{{\mathbb{R}}}
\newcommand{\CC}{{\mathbb{C}}}

\newtheorem{Prop}[theorem]{Proposition}

\newtheorem{num_example}{Example}

\begin{document}
\bibliographystyle{siam}
\title{Matrix equation techniques for certain evolutionary partial differential equations\thanks{Version of \today.}}

\author{Davide Palitta
        }

\institute{ Davide Palitta \at Research Group Computational Methods in Systems
and Control Theory (CSC),
Max Planck Institute for Dynamics of Complex Technical Systems, Sandtorstra\ss{e} 1, 39106 Magdeburg, Germany\\ \email{palitta@mpi-magdeburg.mpg.de}            
}

\date{Received: date / Accepted: date}

\maketitle

\begin{abstract}
We show that the discrete operator stemming from time-space discretization of evolutionary partial differential equations can be represented in terms of a single Sylvester matrix equation. A novel solution strategy that combines projection techniques with the full exploitation of the entry-wise structure of the involved coefficient matrices
is proposed. The resulting scheme is able to efficiently solve problems with a
tremendous number of degrees of freedom while maintaining a low storage demand as illustrated in several numerical examples. 



\end{abstract}




\section{Introduction}
The numerical treatment of partial differential equations (PDEs) often involves a first discretization phase which yields a discrete operator that needs to be inverted.
In general, if a $d$-dimensional operator on a regular domain is discretized with $n$ nodes in each direction,
a common approach consists in writing the discrete problem as a large linear system
\begin{equation}\label{eq.linear_system}
\mathcal{A}u=f,\quad \mathcal{A}\in\mathbb{R}^{n^d\times n^d}, 
\end{equation}
so that well-established procedures, either direct or iterative, can be employed in the solution process.
However, in many cases, the coefficient matrix $\mathcal{A}$ in \eqref{eq.linear_system} is very structured
and a different formulation of the algebraic problem in terms of a matrix equation can be employed. 
The matrix oriented formulation of the algebraic problems arising from the discretization of certain deterministic and stochastic PDEs is not new. See, e.g., \cite{Starke1991,Wachspress1984,Wachspress1963,Powell2009}.  Nevertheless, only in the last decades the development of efficient solvers for large-scale matrix equations allows for a full exploitation of such reformulation also during the solution phase. See, e.g., \cite{Palitta2016,Kressner2009,DAutilia2019,Stoll2015,Breiten2016,Powell2017},
and \cite{Simoncini2016} for a thorough presentation about solvers for linear matrix equations.

In this paper, we discuss time-dependent PDEs and we show that the aforementioned reformulation in terms of a matrix equation can be performed also for this class of operators. The model problem we have in mind is the heat equation
\begin{equation}\label{heat_eq}
 \begin{array}{rlll}
         u_t&=&\Delta u +f,& \quad \text{in }\Omega\times (0,T],\\
         u&=&g,& \quad \text{on } \partial\Omega,\\
         u(x,0)&=&u_0(x),&
        \end{array}
\end{equation}
where $\Omega\subset\mathbb{R}^d$, $d=1,2,3$, is a regular domain. In particular, we specialize some of our results in the case of tensorized spatial domains of the form $\Omega=\bigotimes_{i=1}^d\Omega_i$ . However, the matrix equation formulation we propose in this paper still holds for more general domains $\Omega$.

We discretize the problem \eqref{heat_eq} in both space and time, and, for sake of simplicity, we assume that a finite difference method with a uniform mesh is employed in the space discretization whereas we apply a backward differentiation formula (BDF) of order $s$, $s=1,\ldots,6$, for the discretization in time.

If an ``all-at-once'' approach is considered, the algebraic problem arising from the discretization of \eqref{heat_eq} amounts to a linear system of the form \eqref{eq.linear_system} with $\mathcal{A}\in\mathbb{R}^{n^d\ell\times n^d\ell}$ where $n$ is the number of nodes employed in each of the $d$ space directions, $d$ is the space dimension and $\ell$ is the number of time steps.
As shown in \cite{McDonald2018}, the $n^d\ell\times n^d\ell$ coefficient matrix $\mathcal{A}$ possesses a Kronecker structure. While in 
\cite{McDonald2018} the authors exploit this Kronecker form to design an effective preconditioner for \eqref{eq.linear_system}, we take advantage of the Kronecker structure to reformulate the algebraic problem in terms of a matrix equation and we show how appropriate projection techniques
can be applied for its efficient solution.

The most common approximation spaces used in the solution of matrix equations by projection are the extended Krylov subspace
\begin{equation}\label{def.extended}
\mathbf{EK}_m^\square(A,B):=\mbox{Range}([B,A^{-1}B,AB,\ldots,A^{m-1}B,A^{-m}B]), \quad A\in\mathbb{R}^{n\times n},B\in\mathbb{R}^{n\times p},p\ll n,
\end{equation}
see, e.g., \cite{Simoncini2007,Knizhnerman2011}, and the more general rational Krylov subspace
\begin{equation}\label{def.rational}
\mathbf{K}_m^\square(A,B,\pmb{\xi}):=\mbox{Range}([B,(A-\xi_2I)^{-1}B,\ldots,\prod_{i=2}^m(A-\xi_iI)^{-1}B]), 
\end{equation}
where $\pmb{\xi}=[\xi_2,\ldots,\xi_m]^T\in\CC^{m-1}$. See, e.g., \cite{Druskin2011a,Druskin2011,Druskin2014}. We thus consider
only these spaces in our analysis.

Here is a synopsis of the paper.
Assuming the backward Euler scheme, i.e., a BDF of order 1, is employed for the time integration, in section~\ref{A matrix equation formulation} we show how the all-at-once approach for the solution of \eqref{heat_eq} leads to a Sylvester matrix equation.
An automatic incorporation of the boundary conditions for the matrix equation formulation is illustrated in section~\ref{Imposing the boundary conditions} while in section~\ref{The extended Krylov subspace method} the efficient solution of the obtained algebraic problem is discussed. In particular, in section~\ref{Left projection} we present the new solution procedure for problems where only the space component of the discrete operator is reduced by projection onto a suitable subspace, i.e., we consider problems where the number of time steps $\ell$ is small, say $\ell=\mathcal{O}(10^3)$.
For $d=2,3$, also the stiffness matrix arising from the discretization of the Laplace operator has a Kronecker structure that can be further exploited in the solution process as illustrated in section~\ref{Structured space operators}. 
At each iteration, the projection technique presented in section~\ref{Left projection}
requires the solution of a reduced equation and this task is one of most expensive parts of the entire procedure, especially for large $\ell$. 
In section~\ref{Efficient inner solves} we illustrate a novel strategy that dramatically decreases the cost of such inner solves.
In section~\ref{Multistep methods} we generalize the approach to the case of generic BDFs of order $s$, $s=1,\ldots,6$.
For the sake of simplicity, only the extended Krylov subspace \eqref{def.extended}
is considered in the discussion presented in section~\ref{The extended Krylov subspace method} but in 
section~\ref{The rational Krylov subspace method} we show how to easily adapt our new strategy when the rational Krylov subspace 
\eqref{def.rational} is adopted as approximation space.
The novel framework we present can be employed in the solution of many different PDEs and in section~\ref{The convection-diffusion equation} we describe the solution process in case of time-dependent convection-diffusion equations.
Several results illustrating the potential of our new methodology are reported in section~\ref{Numerical results} while our conclusions are given in section~\ref{Conclusions}.

Throughout the paper we adopt the following notation.
 The matrix inner product is
defined as $\langle X, Y \rangle_F ∶= \mbox{trace}(Y^T X)$ so that the induced norm is $\|X\|_F^2= \langle X, X\rangle_F$. The Kronecker product is denoted by $\otimes$
while the operator $\text{vec}:\mathbb{R}^{n\times n}\rightarrow\mathbb{R}^{n^2}$ is such that $\text{vec}(X)$ is the vector obtained by stacking the columns of the matrix $X$ one on top of each other.
The identity matrix of order $n$ is denoted by $I_n$. The subscript
is omitted whenever the dimension of $I$ is clear from the context. Moreover, $e_i$ is the $i$-th basis vector of the canonical basis of $\mathbb{R}^n$ while $E_i$ denotes the $i$-th block of $q$ columns of an identity matrix whose dimension depends on the adopted approximation space. More precisely, when the extended Krylov subspace~\eqref{def.extended} is employed, $q=2\cdot p$ while $q=p$ when the rational Krylov subspace~\eqref{def.rational} is selected.   
The brackets $[\cdot]$ are used to concatenate matrices of conforming dimensions. In particular, a Matlab-like notation is adopted and $[M,N]$ denotes the matrix obtained by putting $M$ and $N$ one next to the other.
If $w\in\mathbb{R}^{n}$, $\text{diag}(w)$ denotes the $n\times n$ diagonal matrix whose $i$-th diagonal entry corresponds to the $i$-th component of $w$.

Given a suitable space $\mathcal{K}_m$\footnote{$\mathcal{K}_m$ as in \eqref{def.extended} or \eqref{def.rational}.}, we will always assume that a matrix $V_m\in\RR^{n\times r}$, $\mbox{Range}(V_m)=\mathcal{K}_m$, has orthonormal columns and it is full rank so that $\mbox{dim}(\mathcal{K}_m)=r$. Indeed, if this is not the case, deflation strategies to overcome the possible linear dependence of the spanning vectors can be adopted as it is customary in block Krylov methods. See, e.g., \cite[Section 8]{Gutknecht2006}.

\section{A matrix equation formulation}\label{A matrix equation formulation}
Assuming that the backward Euler scheme is employed for the time integration,
if $\widebar \Omega_h=\{\widebar x_{\mathbf{i}_d}\}$, $\widebar x_{\mathbf{i}_d}\in\mathbb{R}^d$, $\mathbf{i}_d=(i_1,\ldots,i_d)^T\in\mathbb{N}^d$, $i_j=1,\ldots,n$ for all $j=1,\ldots,d$,
denotes a uniform discretization of the closed domain $\widebar\Omega$, with $n$ equidistant points in each of the $d$ spatial dimensions, and
the time interval $[0,T]$ is discretized with $\ell+1$ equidistant nodes $\{t_k\}_{k=0,\ldots,\ell}$, then the discretization of~\eqref{heat_eq} leads to 
\begin{equation}\label{discrete_eq}
 \frac{\mathbf{u}_k-\mathbf{u}_{k-1}}{\tau}+K_d\mathbf{u}_k=\mathbf{f}_k,\quad k=1,\ldots,\ell.
\end{equation}
In \eqref{discrete_eq}, $K_d\in\mathbb{R}^{n^d}$ denotes the stiffness matrix arising from the finite difference discretization of the $d$-dimensional negative laplacian on $\widebar \Omega_h$, $\tau=T/\ell$ is the time-step size, $\mathbf{f}_k\in\mathbb{R}^{n^d}$ collects all the space nodal values of $f$ at time $t_k$, namely $f(x_{\mathbf{i}_d},t_k)$ for all $x_{\mathbf{i}_d}\in\widebar \Omega_h$, together with the boundary conditions, while $\mathbf{u}_k$ gathers the approximations to the space nodal values of the solution $u$ at time $t_k$, i.e., $u(x_{\mathbf{i}_d},t_k)$ for all $x_{\mathbf{i}_d}\in\widebar \Omega_h$\footnote{We assume the entries of both $\mathbf{f}_k$ and $\mathbf{u}_k$ to be sorted following a lexicographic order on the multi-index $\mathbf{i}_d$ for all $k=1,\ldots,\ell$.}.

As shown in \cite{McDonald2018}, rearranging the terms in \eqref{discrete_eq} and applying an all-at-once approach, we get the $n^d\ell\times n^d\ell$ linear systems
\begin{equation}\label{all_at_once_system}
 \underbrace{\begin{bmatrix}
  I_{n^d}+\tau K_d & & &\hspace{-1cm} \\
  -I_{n^d} & I_{n^d}+\tau K_d & &\hspace{-1cm} \\
 \hspace{1cm} \ddots & \hspace{1cm} \ddots & & \hspace{-1cm}\\
   &\hspace{-0.8cm}  -I_{n^d} &\hspace{-0.2cm} I_{n^d}+\tau K_d &\hspace{-1cm}\\
 \end{bmatrix}}_{=:\mathcal{A}}\begin{bmatrix}
 \mathbf{u}_1  \\
 \mathbf{u}_2\\
 \vdots\\
 \mathbf{u}_\ell\\
 \end{bmatrix}=\begin{bmatrix}
 \mathbf{u}_0+\tau \mathbf{f}_1 \\
 \tau \mathbf{f}_2 \\
 \vdots\\
 \tau \mathbf{f}_\ell \\
 \end{bmatrix},
\end{equation}
where $\mathbf{u}_0$ collects the space nodal values of the initial condition $u_0$.

The coefficient matrix $\mathcal{A}$ in \eqref{all_at_once_system} can be written as $\mathcal{A}=I_\ell\otimes(I_{n^d}+\tau K_d)-\Sigma_1\otimes I_{n^d}$ where 
$$\Sigma_1=\begin{bmatrix}
            0 & & & \\
            1 & 0 & & \\
            & \ddots & \ddots & \\
            & & 1 & 0 \\
           \end{bmatrix}\in\mathbb{R}^{\ell\times \ell}.
$$
Therefore, if $\mathbf{U}=[\mathbf{u}_1,\ldots,\mathbf{u}_\ell]\in\mathbb{R}^{n^d\times \ell}$, the linear system \eqref{all_at_once_system} can be reformulated as
\begin{equation}\label{Sylv_eq}
 (I_{n^d}+\tau K)\mathbf{U}-\mathbf{U}\Sigma_1^T=\mathbf{u}_0e_1^T+\tau[\mathbf{f}_1,\ldots,\mathbf{f}_\ell].
\end{equation}
Many numerical methods for the efficient solution of the Sylvester matrix equation \eqref{Sylv_eq} can be found in the literature, see, e.g., \cite{Simoncini2016}, and in section~\ref{The extended Krylov subspace method} we present a procedure based on projection.

In what follows we always assume that the matrix
$[\mathbf{f}_1,\ldots,\mathbf{f}_\ell]$ admits a low-rank representation, namely $[\mathbf{f}_1,\ldots,\mathbf{f}_\ell]= F_1F_2^T$, $F_1\in\mathbb{R}^{n^d\times p}$, $F_2\in\mathbb{R}^{\ell\times p}$, $p\ll\min\{n^d,\ell\}$. Roughly speaking, this can be justified by assuming the functions $f$ and $g$ to be \emph{sufficiently smooth} in time so that $\mathbf{f}_k$ does not differ too much from $\mathbf{f}_{k+1}$ if the time-step size $\tau$ is sufficiently small. More precisely, if $\mathbf{f}_k$ contains entries having an analytic extension in an open elliptic disc with foci $0$ and $T$ for all $k$, then the results in \cite[Lemma 2.2]{Kressner2011}
and \cite[Corollary~2.3]{Kressner2011} can be adapted to
 demonstrate an exponential (superexponential in case of entire function) decay in the singular values of $[\mathbf{f}_1,\ldots,\mathbf{f}_\ell]$. This can be done 
by simply transforming the interval $[-1,1]$ used in \cite[Lemma 2.2]{Kressner2011} in the interval $[0,T]$. With this assumption, equation \eqref{Sylv_eq} can be written as
\begin{equation}\label{Sylv_eq2}
 (I_{n^d}+\tau K_d)\mathbf{U}-\mathbf{U}\Sigma_1^T=[\mathbf{u}_0,F_1][e_1,\tau F_2]^T.
\end{equation}

If a finite element method is employed for the space discretization, also a mass matrix $M$ has to be taken into account and the matrix equation we have to deal with has the form
\begin{equation}\label{GenSylv_eq}
 (M+\tau K_d)\mathbf{U}-M\mathbf{U}\Sigma_1^T=[M\mathbf{u}_0,F_1][e_1,\tau F_2]^T.
\end{equation}
See, e.g., \cite{McDonald2018}. The generalized Sylvester equation \eqref{GenSylv_eq} can be easily transformed into a standard Sylvester equation by premultiplying by $M^{-1}$, see, e.g., \cite[Section 7]{Simoncini2016}, and the procedure we are going to present in section~\ref{The extended Krylov subspace method} can be applied to 
$$(I_{n^d}+\tau M^{-1}K_d)\mathbf{U}-\mathbf{U}\Sigma_1^T=[\mathbf{u}_0,M^{-1}F_1][e_1,\tau F_2]^T.
$$

\section{Imposing the boundary conditions}\label{Imposing the boundary conditions}
Before showing how to efficiently solve equation \eqref{Sylv_eq} by projection, we make a step back and illustrate an automatic procedure for including the boundary conditions in the formulation \eqref{Sylv_eq} in case of tensorized spatial domains. For sake of simplicity, we assume $\Omega=(0,1)^d$.

We first consider $d=1$ in \eqref{heat_eq}. 
 The boundary nodes correspond to the entries of index $i$, $i=1,n$, in each column of $\mathbf{U}$. Denoting by $\mathcal{P}_1$ the operator which selects only the boundary nodes, namely its entries are 1 for indexes corresponding to boundary nodes and 0 otherwise, for 1-dimensional problems we have
$$
\mathcal{P}_1=\begin{bmatrix}
             1 & & & &\\
               & 0 & & &\\
               & & \ddots && \\
               & & & & 0 & \\
               & & & & & 1\\
            \end{bmatrix}=e_1e_1^T+e_ne_n^T.
            $$

The operator $I+\tau K_1$ should act as the  identity operator on the space boundary nodes which means that  
\begin{equation}\label{BC_constrain}
 \mathcal{P}_1(I+\tau K_1)=\mathcal{P}_1.
\end{equation}
Therefore, if we define the matrix
\begin{equation}\label{BC_imposed}
\widebar K_1:=\begin{bmatrix}
             1/\tau & & \\
               & \mathring{K}_1 &\\
                & & 1/\tau\\                     
          \end{bmatrix}\in\mathbb{R}^{n\times n},
\end{equation}
we can consider $I_n-\mathcal{P}_1+\tau
\widebar K_1$ in place of $I_n+\tau
K_1$ as left coefficient matrix in \eqref{Sylv_eq}. In \eqref{BC_imposed}, the matrix $\mathring{K}_1\in\mathbb{R}^{(n-2)\times n}$ corresponds to the discrete operator stemming from the selected finite difference scheme and acting only on the interior of $\widebar \Omega_h$. Different choices with respect to the one in \eqref{BC_imposed} can be considered to meet the constrain~\eqref{BC_constrain}. For instance, we can select $\widetilde K_1:=[\underline{0}^T;\mathring K_1; \underline{0}^T]$, $\underline{0}$ the zero vector of length $n$, and consider $I_n+\tau \widetilde K_1$ as coefficient matrix.  
However, such a $\widetilde K_1$ is not suitable for the solution process we are going to present in section~\ref{The extended Krylov subspace method} due to its singularity and the matrix $\widebar K_1$ in \eqref{BC_imposed} is thus preferred.

We now show how to select the right-hand side in \eqref{Sylv_eq} when the coefficient matrix is as in \eqref{BC_imposed}. We have
$$
 \mathcal{P}_1(I_n-\mathcal{P}_1+\tau
\widebar K_1)\mathbf{U}-\mathcal{P}_1\mathbf{U}\Sigma_1^T=\mathcal{P}_1(\mathbf{u}_0e_1^T+\tau[\mathbf{f}_1,\ldots,\mathbf{f}_\ell]),$$
 so that 
 $$
\begin{bmatrix}
  \mathbf{u}_1(1) & \mathbf{u}_2(1)-\mathbf{u}_1(1) & \cdots & \mathbf{u}_\ell(1)-\mathbf{u}_{\ell-1}(1) \\
  0 & 0 & & 0 \\
  \vdots & \vdots & & \vdots\\
  0 & 0& & 0 \\
  \mathbf{u}_1(n) & \mathbf{u}_2(n)-\mathbf{u}_1(n) & \cdots & \mathbf{u}_\ell(n)-\mathbf{u}_{\ell-1}(n) \\
 \end{bmatrix} = \begin{bmatrix}
  u_0(x_1)+\tau \mathbf{f}_1(1) & \tau \mathbf{f}_2(1) & \cdots & \tau \mathbf{f}_\ell(1) \\
  0 & 0 & & 0 \\
  \vdots & \vdots & & \vdots\\
  0 & 0& & 0 \\
  u_0(x_n)+\tau \mathbf{f}_1(n) & \tau \mathbf{f}_2(n) & \cdots & \tau \mathbf{f}_\ell(n) \\
 \end{bmatrix}.  
$$
Therefore, we can set $\mathbf{f}_1(1)=\mathbf{f}_1(n)=0$ whereas $\mathbf{f}_k(j)=(g(x_j,t_k)-g(x_j,t_{k-1}))/\tau$, $k=2,\ldots,\ell,$ $j=1,n$.

A similar approach can be pursued also for 2- and 3-dimensional problems.
In this cases, following the same ordering of the unknowns proposed in \cite{Palitta2016}, it can be shown that the operator selecting the boundary nodes in $\mathbf{U}$ has the form 
$$ \mathcal{P}_2=\mathcal{P}_1\otimes I_n+(I_n-\mathcal{P}_1)\otimes \mathcal{P}_1,\quad
\mathcal{P}_3=\mathcal{P}_1\otimes I_n\otimes I_n+(I_n-\mathcal{P}_1)\otimes \mathcal{P}_1\otimes I_n+(I_n-\mathcal{P}_1)\otimes(I_n-\mathcal{P}_1)\otimes\mathcal{P}_1,$$
for $d=2,3$ respectively.
 
 It is well-known that also $K_d$ possesses a Kronecker structure. In particular, 
$$K_2=K_1\otimes I_n+I_n\otimes K_1,\quad K_3=K_1\otimes I_n\otimes I_n+I_n\otimes K_1\otimes I_n+I_n\otimes I_n\otimes K_1.$$
The most natural choice for imposing the boundary conditions is thus to select
$$\widebar K_2=\widebar K_1\otimes I_n+I_n\otimes\widebar  K_1,\quad \widebar K_3=\widebar K_1\otimes I_n\otimes I_n+I_n\otimes \widebar K_1\otimes I_n+I_n\otimes I_n\otimes \widebar K_1,$$
and use $I_{n^2}-\mathcal{P}_2+\tau \widebar K_2$ and 
$I_{n^3}-\mathcal{P}_3+\tau \widebar K_3$ as coefficient matrices in \eqref{Sylv_eq}. Notice that $I_{n^d}-\mathcal{P}_d=\bigotimes_{i=1}^d(I_n-\mathcal{P}_1)$.

A direct computation shows that 
\begin{equation}\label{BC2_constrain}
\mathcal{P}_2\left(\bigotimes_{i=1}^2(I_n-\mathcal{P}_1)+\tau \widebar K_2\right)=
\mathcal{P}_2+\mathcal{P}_1\otimes (I_n-\mathcal{P}_1)\widebar K_1+ (I_n-\mathcal{P}_1)\widebar K_1\otimes \mathcal{P}_1=\mathcal{P}_2+\mathcal{L}_2, 
\end{equation}
and
\begin{align}\label{BC3_constrain}
\mathcal{P}_3\left(\bigotimes_{i=1}^3(I_n-\mathcal{P}_1)+\tau \widebar K_3\right)=&
\mathcal{P}_3+\left(\mathcal{P}_1\otimes I_n
\otimes I_n\right)\left(I_n\otimes \widebar K_1 \otimes I_n+I_n\otimes I_n\otimes\widebar K_1\right) \notag\\
& +\left((I_n-\mathcal{P}_1)\otimes \mathcal{P}_1\otimes I_n\right)\left(\widebar K_1\otimes I_n  \otimes I_n+I_n\otimes I_n\otimes\widebar K_1\right)\notag\\
&+
\left((I_n-\mathcal{P}_1)\otimes (I_n-\mathcal{P}_1)\otimes \mathcal{P}_1\right)\left(\widebar K_1\otimes I_n  \otimes I_n+I_n\otimes \widebar K_1\otimes I_n\right)\notag\\
=&\mathcal{P}_3+\mathcal{L}_3.
\end{align}
Therefore the extra terms $\mathcal{L}_2$, $\mathcal{L}_3$ in \eqref{BC2_constrain}-\eqref{BC3_constrain} must be taken into account when constructing the right-hand side $\mathbf{u}_0e_1^T+\tau[\mathbf{f}_1,\ldots,\mathbf{f}_\ell]$, and the relation
$$ \mathcal{P}_d\left(\bigotimes_{i=1}^d(I_n-\mathcal{P}_1)+\tau
\widebar K_d\right)\mathbf{U}-\mathcal{P}_d\mathbf{U}\Sigma_1^T=\mathcal{P}_d(\mathbf{u}_0e_1^T+\tau[\mathbf{f}_1,\ldots,\mathbf{f}_\ell]),\quad d=2,3,$$
i.e.,
$$ \mathcal{P}_d\mathbf{U}+\mathcal{L}_d\mathbf{U}-\mathcal{P}_d\mathbf{U}\Sigma_1^T=\mathcal{P}_d(\mathbf{u}_0e_1^T+\tau[\mathbf{f}_1,\ldots,\mathbf{f}_\ell]),\quad d=2,3,$$
must hold. See, e.g., \cite[Section 3]{Palitta2016} for a similar construction. 

After imposing the boundary conditions and recalling the discussion at the end of section~\ref{A matrix equation formulation}, the Sylvester equation we thus need to solve is
\begin{equation}\label{Sylv_eq3}
 \left(\bigotimes_{i=1}^d(I_n-\mathcal{P}_1)+\tau
\widebar K_d\right)\mathbf{U}-\mathbf{U}\Sigma_1^T=[\mathbf{u}_0,F_1][e_1,\tau F_2],\quad d=1,2,3,
\end{equation}
and in the next section we illustrate its efficient solution by projection.

We would like to underline the fact that if $\Omega$ is a general domain, a more involved procedure has to be adopted to impose the boundary conditions in the matrix equation formulation \eqref{Sylv_eq} in general. Indeed, a more complex geometry may no longer allow for a stiffness matrix $K_d$ that can be written in terms of a Kronecker sum so that the \emph{left} coefficient matrix in the Sylvester equation we end up with may have a different expression than the one in \eqref{Sylv_eq3}. Nonetheless, the solution framework we are going to present in the following sections can be still employed with straightforward modifications.

\section{The extended Krylov subspace method}\label{The extended Krylov subspace method}
In this section we show how to effectively solve equation \eqref{Sylv_eq3} by means of the extended Krylov subspace method. An efficient implementation of this algorithm called K-PIK for large-scale Lyapunov equations can be found in \cite{Simoncini2007} whereas its extension to the solution of Sylvester equations has been proposed in \cite{Breiten2016}. In the next section we suppose that the number $\ell$ of time steps is moderate, say $\ell=\mathcal{O}(10^3)$, so that only a left projection, i.e., a reduction of the space discrete operator, has to be performed. See, e.g., \cite[Section 5.2]{Palitta2018} or \cite[Section 4.3]{Simoncini2016} for some details about projection methods for this problem setting.

In section~\ref{Efficient inner solves} we then suppose that a large number of time steps $\ell$ is employed in the time discretization so that a naive solution of the inner problems stemming from our projection technique is not feasible. By exploiting the structure of $\Sigma_1$ we propose a valid remedy to overcome this numerical issue.

\subsection{Left projection}\label{Left projection}
The extended Krylov subspace method constructs an approximation $U_m=V_mY_m\in\mathbb{R}^{n^d\times \ell}$ where the $2m(p+1)$ columns of $V_m$ form 
an orthonormal basis of the extended Krylov subspace $\mathbf{EK}_m^\square(\widebar K_d,[\mathbf{u}_0,F_1])$, so that $\text{Range}(V_m)=\mathbf{EK}_m^\square(\widebar K_d,[\mathbf{u}_0,F_1])$.
Notice that we use only $\widebar K_d$ in the definition of the space instead of the whole coefficient matrix $\bigotimes_{i=1}^d(I_n-\mathcal{P}_1)+\tau \widebar K_d$. Indeed, all the spectral information about the spatial operator are collected in $\widebar K_d$. See, e.g., \cite{Simoncini2010} for a similar strategy in the context of extended Krylov subspace methods for shifted linear systems.

The basis $V_m=[\mathcal{V}_1,\ldots,\mathcal{V}_m]\in\mathbb{R}^{n^d\times 2m(p+1)}$ can be 
constructed by the extended Arnoldi procedure presented in \cite{Simoncini2007} while
the matrix $Y_m\in\mathbb{R}^{2m(p+1)\times \ell}$ can be computed, e.g., by imposing a Galerkin condition on the residual matrix $R_m:=(\bigotimes_{i=1}^d(I_n-\mathcal{P}_1)+\tau \widebar K_d)U_m-U_m\Sigma_1^T-[\mathbf{u}_0, F_1][e_1,\tau F_2]^T$. This Galerkin condition can be written as 
$$V_m^TR_m=0,$$
so that $Y_m$ is the solution of the reduced Sylvester equation
\begin{equation}\label{projected_eq}
(\mathcal{I}_{m}+\tau T_m)Y_m-Y_m\Sigma_1^T=E_1\pmb{\gamma} [e_1,\tau F_2]^T,
\end{equation}
where $T_m=V_m^T\widebar K_dV_m$, $\mathcal{I}_{m}=V_m^T(\bigotimes_{i=1}^d(I_n-\mathcal{P}_1))V_m$ and $[\mathbf{u}_0, F_1]=V_1\pmb{\gamma}$, $\pmb{\gamma}\in\mathbb{R}^{2(p+1)\times(p+1)}$. In exact arithmetic, the matrix $T_m$ can be cheaply computed by the recursion formulas presented in \cite{Simoncini2007}. However, 
from our numerical experience, computing an explicit projection of $ \widebar K_d$
leads to a better representation of the boundary conditions in the projected problem \eqref{projected_eq}, and thus in the solution $U_m$ as well, in spite of a moderate computational extra cost. The recursion formulas in \cite{Simoncini2007} probably suffers the presence of $1/\tau$ in the definition~\eqref{BC_imposed} of $\widebar K_1$, especially for very small $\tau$.  

An explicit projection has to be performed also to construct $\mathcal{I}_m$. However, the particular structure of $\bigotimes_{i=1}^d(I_n-\mathcal{P}_1)$ makes this task affordable in terms of number of operations.
For instance, if $d=1$, we have
$$\mathcal{I}_{m}=V_m^T(I_n-\mathcal{P}_1)V_m=
V_m^T(I_n-e_1e_1^T-e_ne_n^T)V_m=I_{2m(p+1)}-(V_m^Te_1)(V_m^Te_1)^T-(V_m^Te_n)(V_m^Te_n)^T,
$$
so that only the small matrices $(V_m^Te_1)(V_m^Te_1)^T$, $(V_m^Te_n)(V_m^Te_n)$ have to be computed. Moreover, at the following iteration, $(V_{m+1}^Te_i)=[V_{m}^Te_i; \mathcal{V}_{m+1}^Te_i]$ and this structure can be exploited to further reduce the cost of computing $\mathcal{I}_m$. A similar discussion shows that the computation of $\mathcal{I}_m$ is a minor cost also for $d=2,3$.  
%

Due to its small dimension, equation \eqref{projected_eq} can be solved by means of general-purposed dense solvers for Sylvester equations like the Bartels-Stewart method \cite{Bartels1972} or the Hessenberg-Schur method presented in \cite{Golub1979} which may be particularly appealing in our context due to the lower Hessenberg pattern of $\Sigma_1^T$. See also \cite[Section 3]{Benner2011}. However, the structure of \eqref{projected_eq} allows for a cheaper alternative. If $Y_m=[y_1,\ldots,y_\ell]$,
then equation \eqref{projected_eq} can be written as
$$[(\mathcal{I}_{m}+\tau T_m)y_1,(\mathcal{I}_{m}+\tau T_m)y_2-y_1,\ldots,(\mathcal{I}_{m}+\tau T_m)y_{\ell}-y_{\ell-1}]=E_1\pmb{\gamma} [e_1,\tau F_2]^T.$$
Since $\bigotimes_{i=1}^d(I_n-\mathcal{P}_1)+\tau \widebar K_d$ is positive definite, also $\mathcal{I}_{m}+\tau T_m$ is positive definite and thus invertible for every $m$.
Writing the relation above column-wise we get
\begin{equation}\label{Y_comp}
\begin{array}{rll}
 y_1&=&(\mathcal{I}_{m}+\tau T_m)^{-1}E_1\pmb{\gamma} [e_1,\tau F_2]^Te_1,\\ 
 &&\\
y_j&=&(\mathcal{I}_{m}+\tau T_m)^{-1}\left(E_1\pmb{\gamma} [e_1,\tau F_2]^Te_j+y_{j-1}\right), \quad j=2,\ldots,\ell.
\end{array} 
\end{equation}
This means that the columns of $Y_m$ can be computed by sequentially solving $\ell$ small linear systems with the same coefficient matrix
 $\mathcal{I}_{m}+\tau T_m$ whose factorization can be computed only once at each iteration. 

Once $Y_m$ is computed, it is easy to show that the Frobenius norm of the residual matrix $R_m$ can be cheaply evaluated as
\begin{equation}\label{res_norm_comp}
\|R_m\|_F=\tau\|E_{m+1}^T\underline{T}_mY_m\|_F, 
\end{equation}
where $\underline{T}_m=V_{m+1}^T\widebar K_dV_m$. See, e.g., \cite[Section 5.2]{Palitta2018}.


In Algorithm~\ref{EK_algorithm} the extended Krylov subspace method for equation \eqref{Sylv_eq3} is summarized.

\setcounter{AlgoLine}{0}
\begin{algorithm}
\caption{Extended Krylov subspace method for \eqref{Sylv_eq3} - left projection.\label{EK_algorithm}}
\SetKwInOut{Input}{input}\SetKwInOut{Output}{output}
\Input{$\widebar K_d\in\mathbb{R}^{n^d\times n^d},$ $\Sigma_1\in\mathbb{R}^{\ell\times \ell}$, $\mathbf{u}_0\in\mathbb{R}^{n^d}$, $F_1\in\mathbb{R}^{n^d\times p}$, $F_2\in\mathbb{R}^{\ell\times p}$, $m_{\max}$, $\epsilon>0$, $\tau>0.$}
\Output{$V_m$, $Y_m$ s.t. $U_m=V_mY_m\approx\mathbf{U}$ approximate solution to \eqref{Sylv_eq3}.}
\BlankLine
\nl Compute $\delta=\|[\mathbf{u}_0, F_1][e_1,\tau F_2]^T\|_F$ \label{initial_residual_norm}\\
  \nl  Perform economy-size QR, $[\mathbf{u}_0,F_1,\widebar K_d^{-1}[\mathbf{u}_0,F_1]]=[\mathcal{V}_1^{(1)},\mathcal{V}_1^{(2)}][
  {\pmb \gamma}, {\pmb \theta}]$, ${\pmb \gamma},{\pmb \theta}\in\RR^{2(p+1)\times (p+1)}$\\                                                                                          \nl Set $V_1= [\mathcal{V}_1^{(1)},\mathcal{V}_1^{(2)}]$ \\ 
  \For{$m=1, 2,\dots,$ till $m_{\max}$}{
  \nl Compute next basis block $\mathcal{V}_{m+1}$ as in \cite{Simoncini2007} and set $V_{m+1}=[V_{{m}},\mathcal{V}_{m+1}]$ \\
  \nl  Update $T_{m}=V_{m}^T\widebar K_dV_{m}$ and $\mathcal{I}_{m}=V_m^T\left(\bigotimes_{i=1}^d(I_n-\mathcal{P}_1)\right)V_m$\\
  \nl Compute  $Y_m$ as in \eqref{Y_comp} \\
  \If{$\tau\|E_{m+1}^T\underline{T}_{m}Y_m\|_F\leq\delta\cdot \epsilon$}{ 
\nl \textbf{Return} $V_m$ and $Y_m$  }
}
\end{algorithm}

Notice that the initial residual norm $\|[\mathbf{u}_0, F_1][e_1,\tau F_2]^T\|_F$ in line~\ref{initial_residual_norm} of Algorithm~\ref{EK_algorithm} can be computed at low cost exploiting the properties of the Frobenius norm and the trace operator. Indeed,
$$\begin{array}{rll}
   \delta^2&=& \|[\mathbf{u}_0, F_1][e_1,\tau F_2]^T\|_F^2=\|\mathbf{u}_0e_1^T\|_F^2+\tau^2\|F_1 F_2^T\|_F^2+2\tau\langle \mathbf{u}_0e_1^T,
   F_1 F_2^T\rangle_F\\
   &&\\
   &=&\mathbf{u}_0^T\mathbf{u}_0+\tau^2\cdot\text{trace}((F_1^TF_1)(F_2^TF_2))+2\tau\mathbf{f}_1^T\mathbf{u}_0.
  \end{array}
$$

In many cases the dimension of the final space
$\mathbf{EK}_m^\square(\widebar K_d,[\mathbf{u}_0,F_1])$, namely the number of columns of $V_m$, turns out to be much smaller than $\ell$. See section~\ref{Numerical results}. Therefore, to reduce the memory demand of Algorithm~\ref{EK_algorithm}, we suggest to store only $V_m$ and $Y_m$ and not to explicitly assemble the solution matrix $U_m=V_mY_m\in\mathbb{R}^{n^d\times \ell}$. If desired, one can access to the computed approximation to the solution $u$ at time $t_k$ by simply performing $V_m(Y_me_k)$.

\subsubsection{Structured space operators}\label{Structured space operators}

As already mentioned,
for 2- and 3-space-dimensional problems, i.e., \eqref{heat_eq} with $d=2,3$, also the stiffness matrix $\widebar K_d$ possesses a Kronecker structure. See section~\ref{Imposing the boundary conditions}.

In principle, one can apply the strategy proposed in section~\ref{Left projection} and build the space $\mathbf{EK}_m^\square(\widebar K_d,[\mathbf{u}_0,F_1])$. However, if $u_0$, $f$ and $g$ in \eqref{heat_eq} are separable functions in the space variables, the Kronecker structure of $\widebar K_2$ and $\widebar K_3$ can be exploited in the basis construction. More precisely, only $d$ subspaces of $\mathbb{R}^n$ can be computed instead of one subspace of $\mathbb{R}^{n^d}$ leading to remarkable reductions in both the computational cost and the storage demand of the overall solution process. See, e.g., \cite{Kressner2009}.
The structure we study in this section is sometimes referred to as \emph{Laplace-like}
structure. Such a structure is at the basis of the \emph{tensorized} Krylov approach presented in \cite{Kressner2009} but it has been exploited also in \cite{Mach2011} to derive an ADI iteration tailored to certain high dimensional problems.

We first assume $d=2$ and then extend the approach to the case of $d=3$. If $\widebar\Omega_h$ consists in $n$ equidistant points in each direction $(x_i,y_j)$, $i,j=1,\ldots,n$, and $u_0=\phi_{u_0}(x)\psi_{u_0}(y)$, then we can write
  $$\mathbf{u}_0=\pmb{\phi}_{u_0}\otimes\pmb{\psi}_{u_0},$$
  where $\pmb{\phi}_{u_0}=[\phi_{u_0}(x_1),\ldots,\phi_{u_0}(x_n)]^T$, and $\pmb{\psi}_{u_0}=[\psi_{u_0}(y_1),\ldots,\psi_{u_0}(y_n)]^T$.

  Similarly, if $f=\phi_f(x,t)\psi_f(y,t)$, 
  $g=\phi_g(x,t)\psi_g(y,t)$, a generic column $\mathbf{f}_k$ of the right-hand side in \eqref{Sylv_eq} can be written as
  $$\mathbf{f}_k=\pmb{\phi}_{f,k}\otimes\pmb{\psi}_{f,k}+\pmb{\phi}_{g,k}\otimes\pmb{\psi}_{g,k},$$
  with 
  $$\begin{array}{ll}
  \pmb{\phi}_{f,k}=[\phi_f(x_1,t_k),\ldots,\phi_f(x_n,t_k)]^T,&
  \pmb{\psi}_{f,k}=[\psi_f(y_1,t_k),\ldots,\psi_f(y_n,t_k)]^T,\\\pmb{\phi}_{g,k}=[\phi_g(x_1,t_k),\ldots,\phi_g(x_n,t_k)]^T,&\pmb{\psi}_{g,k}=[\psi_g(y_1,t_k),\ldots,\psi_g(y_n,t_k)]^T.
   \end{array}
$$
  We further assume that the low-rank factorization $[\mathbf{f}_1,\ldots,\mathbf{f}_\ell]= F_1F_2^T$, $F_1\in\mathbb{R}^{n^2\times p}$, $F_2\in\mathbb{R}^{\ell\times p}$, $p\ll\ell$, is such that the separability features of the functions $f$ and $g$ are somehow preserved. In other words, we assume that we can write 
  $$[\mathbf{f}_1,\ldots,\mathbf{f}_\ell]= (\Phi_f\otimes\Psi_f)F_2^T,$$
  where 
  $\Phi_f\in\mathbb{R}^{n\times q}$, $\Psi_f\in\mathbb{R}^{n\times r}$, $qr=p$. Notice that this construction is not hard to meet in practice. See, e.g., section~\ref{Numerical results}.
 
 With the assumptions above, it has been shown in \cite{Kressner2009} how the construction of a \emph{tensorized} Krylov subspace is very convenient. In particular, we can compute the space  
 $\mathbf{EK}_m^\square(\widebar K_1,[\pmb{\phi}_{u_0},\Phi_f])\otimes \mathbf{EK}_m^\square(\widebar K_1,[\pmb{\psi}_{u_0},\Psi_f])$ instead of
 $\mathbf{EK}_m^\square(\widebar K_2,[\mathbf{u}_0,F_1])$.
 
 The construction of $\mathbf{EK}_m^\square(\widebar K_1,[\pmb{\phi}_{u_0},\Phi_f])$, $\mathbf{EK}_m^\square(\widebar K_1,[\pmb{\psi}_{u_0},\Psi_f])$ is very advantageous in terms of both number of operations and memory requirements compared to the computation of $\mathbf{EK}_m^\square(\widebar K_2,[\mathbf{u}_0,F_1])$. For instance, only multiplications and solves with the $n\times n$ matrix $\widebar K_1$ are necessary while the orthogonalization procedures only involves vectors of length $n$. Moreover, at iteration $m$, we need to store the two matrices $Q_m\in\mathbb{R}^{n\times 2m(q+1)}$, $ \text{Range}( Q_m)=\mathbf{EK}_m^\square(\widebar K_1,[\pmb{\phi}_{u_0},\Phi_f])$,  and $W_m\in\mathbb{R}^{n\times 2m(r+1)}$, $\text{Range}(W_m)=\mathbf{EK}_m^\square(\widebar K_1,[\pmb{\psi}_{u_0},\Psi_f])$, so that only $2m(q+r+2)$ vectors of length $n$ are allocated instead of the $2m(p+1)$ vectors of length $n^2$ the storage of $V_m$ requires.
 Moreover, the construction of the bases $W_m$ and $Q_m$ can be carried out in parallel.
 
 Even if we construct the matrices $W_m$ and $Q_m$ instead of $V_m$, the main framework of the extended Krylov subspace method remains the same. We look for an approximate solution of the form $U_m=(W_m\otimes Q_m)Y_m$ where the $4m^2(q+1)(r+1)\times \ell$ matrix $Y_m$ is computed by imposing a Galerkin condition on the residual matrix $R_m=\left(\bigotimes_{i=1}^2(I_n-\mathcal{P}_1)+\tau(\widebar K_1\otimes I_n+I_n\otimes \widebar K_1)\right)(W_m\otimes Q_m)Y_m-(W_m\otimes Q_m)Y_m\Sigma_1^T-[\pmb{\phi}_{u_0}\otimes\pmb{\psi}_{u_0},\Phi\otimes\Psi][e_1,\tau F_2]^T$. Such Galerkin condition can be written as
 $$(W_m^T\otimes Q_m^T)R_m=0,$$
 so that $Y_m$ is the solution of the reduced Sylvester equation 
 \begin{equation}\label{Y_m_2d}
\left(\mathcal{I}_m\otimes\mathcal{J}_m+\tau(T_m\otimes I_{2m(q+1)}+I_{2m(p+1)}\otimes H_m)\right)Y_m -Y_m\Sigma_1^T=(E_1\pmb{\alpha}\otimes E_1\pmb{\beta})[e_1,\tau F_2]^T,  
 \end{equation}
 where $T_m=W_m^T\widebar K_1W_m$, $H_m=Q_m^T\widebar K_1Q_m$,
 $\mathcal{I}_m=W_m^T(I_n-\mathcal{P}_1)W_m$, $\mathcal{J}_m=Q_m^T(I_n-\mathcal{P}_1)Q_m$,
 $[\pmb{\phi}_{u_0},\Phi_f]=Q_1\pmb{\alpha}$, $\pmb{\alpha}\in\mathbb{R}^{2(q+1)\times(q+1)}$ and $[\pmb{\psi}_{u_0},\Psi_f]=W_1\pmb{\beta}$, $\pmb{\beta}\in\mathbb{R}^{2(r+1)\times(r+1)}$.
 
  As before, the $\ell$ columns of $Y_m$ can be computed by solving $\ell$ linear systems with the same coefficient matrix $\mathcal{I}_m\otimes\mathcal{J}_m+\tau(T_m\otimes I_{2m(q+1)}+I_{2m(p+1)}\otimes H_m)$. 
 
The cheap residual norm computation 
\eqref{res_norm_comp} has not a straightforward counterpart of the form $\|R_m\|_F=\tau\|E_{m+1}^T(\underline{T}_m\otimes I_{2m(q+1)}+I_{2m(r+1)}\otimes \underline{H}_m)Y_m \|_F$, $\underline{T}_m=W_{m+1}^T\widebar K_1W_m$, $\underline{H}_m=Q_{m+1}^T\widebar K_1Q_m$,
in our current setting. A different though cheap procedure for computing the residual norm at low cost is derived in the next proposition.
 \begin{Prop}\label{res_norm_comp_structured_op}
  At the $m$-th iteration of the extended Krylov subspace method, the residual matrix $R_m=\left(\bigotimes_{i=1}^2(I_n-\mathcal{P}_1)+\tau(\widebar K_1\otimes I_n+I_n\otimes \widebar K_1)\right)(W_m\otimes Q_m)Y_m-(W_m\otimes Q_m)Y_m\Sigma_1^T-[\pmb{\phi}_{u_0}\otimes\pmb{\psi}_{u_0},\Phi_f\otimes\Psi_f][e_1,\tau F_2]^T$ is such that 
  \begin{equation}\label{res_norm_comp_2d}
   \|R_m\|_F^2=\tau^2\left(\|\left(E_{m+1}^T\underline{T}_m\otimes I_{2m(q+1)}\right)Y_m\|_F^2+\|\left(I_{2m(r+1)}\otimes E_{m+1}^T\underline{H}_m\right)Y_m \|_F^2\right),
  \end{equation}
where $\underline{T}_m:=W_{m+1}^T\widebar K_1W_m$ and 
$\underline{H}_m:=Q_{m+1}^T\widebar K_1Q_m$.
 \end{Prop}
\begin{proof}
 If $Q_m=[\mathcal{Q}_1,\ldots,\mathcal{Q}_m]$, $\mathcal{Q}_i\in\mathbb{R}^{n\times 2(q+1)}$, $W_m=[\mathcal{W}_1,\ldots,\mathcal{W}_m]$, $\mathcal{W}_i\in\mathbb{R}^{n\times 2(r+1)}$, for the extended Krylov subspaces $\mathbf{EK}_m^\square(\widebar K_1,[\pmb{\phi}_{u_0},\Phi_f]),$ $\mathbf{EK}_m^\square(\widebar K_1,[\pmb{\psi}_{u_0},\Psi_f])$ the Arnoldi relations 
 $$\widebar K_1Q_m=Q_mH_m+\mathcal{Q}_{m+1}E_{m+1}^T\underline{H}_m,$$ 
 and
 $$\widebar K_1W_m=W_mT_m+\mathcal{W}_{m+1}E_{m+1}^T\underline{T}_m,$$ 
 hold. Since $Y_m$ solves \eqref{Y_m_2d}, we have 
 \begin{align*} 
    R_m  = & \left(\bigotimes_{i=1}^2(I_{n}-\mathcal{P}_1)+\tau(\widebar K_1\otimes I_n+I_n\otimes \widebar K_1)\right)(W_m\otimes Q_m)Y_m-(W_m\otimes Q_m)Y_m\Sigma_1^T\\
    &-[\pmb{\phi}_{u_0}\otimes\pmb{\psi}_{u_0},\Phi_f\otimes\Psi_f][e_1,\tau F_2]^T \\
    =& (W_m\otimes Q_m)\left(\left(\mathcal{I}_m\otimes\mathcal{J}_m+\tau(T_m\otimes I_{2m(q+1)}+I_{2m(p+1)}\otimes H_m)\right)Y_m-Y_m\Sigma_1^T-(E_1\pmb{\alpha}\otimes E_1\pmb{\beta})[e_1,\tau F_2]^T\right)\\
    &+\tau\left(\mathcal{W}_{m+1}E_{m+1}^T\underline{T}_m\otimes Q_m+W_m\otimes \mathcal{Q}_{m+1}E_{m+1}^T\underline{H}_m\right)Y_m\\
    =&\tau\left(\mathcal{W}_{m+1}E_{m+1}^T\underline{T}_m\otimes Q_m+W_m\otimes \mathcal{Q}_{m+1}E_{m+1}^T\underline{H}_m\right)Y_m.
    \end{align*}

 Therefore,
 {\small
 $$\begin{array}{rll}
   \| R_m\|_F^2 & = & \tau^2\|\left(\mathcal{W}_{m+1}E_{m+1}^T\underline{T}_m\otimes Q_m+W_m\otimes \mathcal{Q}_{m+1}E_{m+1}^T\underline{H}_m\right)Y_m\|_F^2\\
   &&\\
   &=& \tau^2\left(\|(\mathcal{W}_{m+1}E_{m+1}^T\underline{T}_m\otimes Q_m)Y_m\|_F^2+\|(W_m\otimes \mathcal{Q}_{m+1}E_{m+1}^T\underline{H}_m)Y_m\|_F^2\right.\\
   &&+\left.\langle(\mathcal{W}_{m+1}E_{m+1}^T\underline{T}_m\otimes Q_m)Y_m,(W_m\otimes\mathcal{Q}_{m+1}E_{m+1}^T\underline{H}_m)Y_m\rangle_F\right)
   \\
   &&\\
   &=&\tau^2\left(\|(\mathcal{W}_{m+1}\otimes Q_m)(E_{m+1}^T\underline{T}_m\otimes I_{2m(q+1)})Y_m\|_F^2+\|(W_m\otimes \mathcal{Q}_{m+1})(I_{2m(r+1)}\otimes E_{m+1}^T\underline{H}_m)Y_m\|_F^2\right)\\
   &&\\
   &=&\tau^2\left(\|(E_{m+1}^T\underline{T}_m\otimes I_{2m(q+1)})Y_m\|_F^2+\|(I_{2m(r+1)}\otimes E_{m+1}^T\underline{H}_m)Y_m\|_F^2\right),
   \end{array}
$$}
 where we have exploited the orthogonality of the bases. 
\end{proof}

The variant of Algorithm~\ref{EK_algorithm} that benefits from the separable structure of the data is summarized in Algorithm~\ref{EK_algorithm2}.
 
\setcounter{AlgoLine}{0}
\begin{algorithm}
\caption{Extended Krylov subspace method for \eqref{Sylv_eq3} - left projection, separable data, $d=2$.\label{EK_algorithm2}}
\SetKwInOut{Input}{input}\SetKwInOut{Output}{output}
\Input{$\widebar K_1\in\mathbb{R}^{n\times n},$ $\Sigma_1\in\mathbb{R}^{\ell\times \ell}$, $\pmb{\phi}_{u_0},\pmb{\psi}_{u_0}\in\mathbb{R}^{n}$, $\Phi_f\in\mathbb{R}^{n\times q}$, $\Psi_f\in\mathbb{R}^{n\times r}$, $F_2\in\mathbb{R}^{\ell\times p}$, $m_{\max}$, $\epsilon>0$, $\tau>0.$}
\Output{$Q_m$, $W_m$, $Y_m$ s. t. $U_m=(W_m\otimes Q_m)Y_m\approx\mathbf{U}$ approximate solution to \eqref{Sylv_eq3}.}
\BlankLine
\nl Compute $\delta=\|[\pmb{\phi}_{u_0}\otimes\pmb{\psi}_{u_0}, \Phi_f\otimes\Psi_f][e_1,\tau F_2]^T\|_F$ \label{initial_residual_norm2}\\
  \nl  Perform economy-size QR, $[\pmb{\phi}_{u_0},\Phi_f,\widebar K_1^{-1}[\pmb{\phi}_{u_0},\Phi_f]]=[\mathcal{Q}_1^{(1)},\mathcal{Q}_1^{(2)}][
  {\pmb \alpha}, {\pmb \theta}]$, ${\pmb \alpha},{\pmb \theta}\in\RR^{2(q+1)\times (q+1)}$,
  $[\pmb{\psi}_{u_0},\Psi_f,\widebar K_1^{-1}[\pmb{\psi}_{u_0},\Psi_f]]=[\mathcal{W}_1^{(1)},\mathcal{W}_1^{(2)}][
  {\pmb \beta}, {\pmb \xi}]$, ${\pmb \beta},{\pmb \xi}\in\RR^{2(r+1)\times (r+1)}$  
  \\                                                                                          \nl Set $Q_1= [\mathcal{Q}_1^{(1)},\mathcal{Q}_1^{(2)}]$ and $W_1= [\mathcal{W}_1^{(1)},\mathcal{W}_1^{(2)}]$\\ 
  \For{$m=1, 2,\dots,$ till $m_{\max}$}{
  \nl Compute next basis blocks $\mathcal{Q}_{m+1}$, $\mathcal{W}_{m+1}$ as in \cite{Simoncini2007} and set $Q_{m+1}=[Q_{{m}},\mathcal{Q}_{m+1}]$,
  $W_{m+1}=[W_{{m}},\mathcal{W}_{m+1}]$\\
  \nl  Update $H_{m}=Q_{m}^TK_1Q_{m}$, $T_{m}=W_{m}^TK_1W_{m}$, $\mathcal{I}_m=W_m^T(I_n-\mathcal{P}_1)W_m$ and $\mathcal{J}_m=Q_m^T(I_n-\mathcal{P}_1)Q_m$\\
  \nl Compute  $Y_m$ as in \eqref{Y_m_2d} \\
  \If{$\tau\sqrt{\|\left(E_{m+1}^T\underline{T}_m\otimes I_{2m(q+1)}\right)Y_m\|_F^2+\|\left(I_{2m(r+1)}\otimes E_{m+1}^T\underline{H}_m\right)Y_m \|_F^2}\leq\delta\cdot \epsilon$}{ 
\nl \textbf{Return} $Q_m$, $W_m$ and $Y_m$}
}
\end{algorithm}

Once again, the Frobenius norm $\delta$ at the beginning of Algorithm~\ref{EK_algorithm2} can be cheaply computed by exploiting both the low-rank and the Kronecker structure
of $[\pmb{\phi}_{u_0}\otimes\pmb{\psi}_{u_0}, \Phi_f\otimes\Psi_f][e_1,\tau F_2]^T$.

Having $Q_m$, $W_m$ and $Y_m$ at hand, we can compute the approximation to the solution $u$ at time $t_k$ by performing $\text{vec}(Q_m\widebar Y_{m,k}W_m^T)$ where $\widebar Y_{m,k}\in\mathbb{R}^{2m(q+1)\times 2m(r+1)}$ is such that $\text{vec}(\widebar Y_{m,k})=Y_me_k$.

For 3-space-dimensional problems with separable data we can follow the same approach. If,
$$
\mathbf{u}_0=\pmb{\phi}_{u_0}\otimes\pmb{\psi}_{u_0}\otimes\pmb{\upsilon}_{u_0}, \text{ and } [\mathbf{f}_1,\ldots,\mathbf{f}_\ell]= (\Phi_f\otimes\Psi_f\otimes \Upsilon_f)F_2^T,
$$
then we can compute the subspaces
$\mathbf{EK}_m^\square(\widebar K_1,[\pmb{\phi}_{u_0},\Phi_f])$, $\mathbf{EK}_m^\square(\widebar K_1,[\pmb{\psi}_{u_0},\Psi_f])$
 and $\mathbf{EK}_m^\square(\widebar K_1,[\pmb{\upsilon}_{u_0},\Upsilon_f])$ instead of $\mathbf{EK}_m^\square(\widebar K_3,[\mathbf{u}_0,F_1])$. The derivation of the method follows the same exact steps as before along with straightforward technicalities and we thus omit it here.

\subsection{Efficient inner solves}\label{Efficient inner solves}

One of the computational bottlenecks of Algorithm~\ref{EK_algorithm} is the solution of the inner problems~\eqref{projected_eq}. For large $\ell$, this becomes the most expensive step of the overall solution process. Therefore, especially for problems that require a fine time grid, a more computational appealing alternative to the solution of the $\ell$ linear systems in \eqref{Y_comp} must be sought. 

In principle, one may think to generate a second approximation space in order to reduce also the time component of the discrete operator in \eqref{Sylv_eq3}, in agreement with standard procedures for Sylvester equations. See, e.g., \cite[Section 4.4.1]{Simoncini2016}. However, no extended Krylov subspace can be generated by $\Sigma_1$ due to its singularity. A different option may be to generate the polynomial Krylov subspace $\mathbf{K}_k^\square(\Sigma_1,[e_1,F_2])=\text{Range}\left(\left[[e_1,F_2],\Sigma_1[e_1,F_2],\ldots, \Sigma_1^{k-1}[e_1,F_2]\right]\right)$. Nevertheless, this space is not very informative as $\text{Ker}(\Sigma_1)=\text{span}\{e_1\}$ and the action of $\Sigma_1$ on a vector $v=(v_1,\ldots,v_l)^T\in\mathbb{R}^\ell$ only consists in a permutation of its components of the form $\Sigma_1v=(0,v_1,\ldots,v_{\ell-1})^T$ so that $\Sigma_1^kv=(0,\ldots,0,v_1,\ldots,v_{\ell-k})^T$, $k\leq \ell$.
Alternatively, one can try to apply an ADI iteration tailored to Sylvester equations \cite{Benner2014}. However, the shift selection for the \emph{right} coefficient matrix $\Sigma_1^T$ may be tricky.

The matrix $\Sigma_1$ is such that
\begin{equation}\label{circulant}
 \Sigma_1=C_1-e_1e_\ell^T,\quad C_1=\begin{bmatrix}
            0 & & & 1\\
            1 & 0 & & \\
            & \ddots & \ddots & \\
            & & 1 & 0 \\
           \end{bmatrix}\in\mathbb{R}^{\ell\times\ell}.
\end{equation}
 This relation has been exploited in \cite{McDonald2018} to design an effective preconditioner for \eqref{eq.linear_system}.
 
 We can use \eqref{circulant} to transform equation \eqref{Sylv_eq3} in a \emph{generalized Sylvester equation} of the form
$$
\left(\bigotimes_{i=1}^d(I_n-\mathcal{P}_1)+\tau
\widebar K_d\right)\mathbf{U}-\mathbf{U}C_1^T+\mathbf{U}e_\ell e_1^T=[\mathbf{u}_0,F_1][e_1,\tau F_2],
$$
and the extended Krylov subspace $\mathbf{EK}_k^\square(C_1,[e_1,F_2])$ may be employed in the solution process thanks to the low rank of the term $\mathbf{U}e_\ell e_1^T$ as proposed in \cite{Jarlebring2018}. However, useful spectral information are difficult to generate also in $\mathbf{EK}_k^\square(C_1,[e_1,F_2])$ since $C_1$ is a permutation matrix.

We take advantage of the relation \eqref{circulant} in a different manner.
At each iteration $m$ of Algorithm~\ref{EK_algorithm}, the projected equation~\eqref{Y_comp} can be written as
\begin{equation}\label{Y_m_circulant}
(\mathcal{I}_m+\tau T_m)Y_m-Y_mC_1^T+Y_me_\ell e_1^T=E_1\pmb{\gamma}[e_1,\tau F_2]^T. 
\end{equation}
Since the Krylov space dimension is assumed to be small, we can compute the eigendecomposition of the coefficient matrix $\mathcal{I}_m+\tau T_m$, namely $\mathcal{I}_m+\tau T_m=S_m\Lambda_mS_m^{-1}$, $\Lambda_m=\text{diag}(\lambda_1,\ldots,\lambda_{2m(p+1)})$ whereas,
thanks to its circulant structure, $C_1$ can be diagonalized by the fast Fourier transform (FFT), i.e., $C_1=\mathcal{F}^{-1}\Pi \mathcal{F}$, $\Pi=\text{diag}(\mathcal{F}(C_1e_1))$, where $\mathcal{F}$ denotes the discrete Fourier transform matrix. See, e.g., \cite[Equation (4.7.10)]{Golub2013}.


Pre and postmultiplying equation \eqref{Y_m_circulant} by $S_m^{-1}$ and $\mathcal{F}^{T}$ respectively, we get
\begin{equation}\label{Y_m_circulant2}
\Lambda_m\widetilde Y_m-\widetilde Y_m\Pi+\widetilde Y_m (\mathcal{F}^{-T}e_\ell)(\mathcal{F} e_1)^T=S_m^{-1}E_1\pmb{\gamma}(\mathcal{F}[e_1,\tau F_2])^T, \quad \widetilde Y_m:=S_m^{-1}Y_m\mathcal{F}^T. 
\end{equation}
The Kronecker form of equation \eqref{Y_m_circulant2} is 
$$\left(I_\ell\otimes\Lambda_m -\Pi\otimes I_{2m(p+1)}+(\mathcal{F}e_1\otimes I_{2m(p+1)})(\mathcal{F}^{-T}e_\ell\otimes I_{2m(p+1)})^T\right)\text{vec}(\widetilde Y_m)=\text{vec}(S_m^{-1}E_1\pmb{\gamma}(\mathcal{F}[e_1,\tau F_2])^T).$$
Denoting by $L:=I_\ell\otimes\Lambda_m -\Pi\otimes I_{2m(p+1)}\in\mathbb{R}^{2m(p+1)\ell\times 2m(p+1)\ell}$, $M:=\mathcal{F}e_1\otimes I_{2m(p+1)},$ $N:=\mathcal{F}^{-T}e_\ell\otimes I_{2m(p+1)}\in\mathbb{R}^{2m(p+1)\ell\times 2m(p+1)}$, and applying the Sherman-Morrison-Woodbury formula \cite[Equation (2.1.4)]{Golub2013} we can write
{\small
\begin{equation}\label{Y_m_circulant3}
\text{vec}(\widetilde Y_m)=L^{-1}\text{vec}(S_m^{-1}E_1\pmb{\gamma}(\mathcal{F}[e_1,\tau F_2])^T)-L^{-1}M(I_{2m(p+1)}+N^TL^{-1}M)^{-1}N^TL^{-1}\text{vec}(S_m^{-1}E_1\pmb{\gamma}(\mathcal{F}[e_1,\tau F_2])^T).
\end{equation}
}
With $\widetilde Y_m$ at hand, we can recover $Y_m$ by simply performing $Y_m=S_m\widetilde Y_m \mathcal{F}^{-T}$.

We are thus left with deriving a strategy for the computation of $\widetilde Y_m$ that should not require 
the explicit construction of $L$, $M$ and $N$ to be efficient. In what follows $\odot$ denotes the Hadamard (element-wise) product. 

Denoting by $\mathcal{H}\in\mathbb{R}^{2m(p+1)\times \ell}$ the matrix whose $(i,j)$-th element is given by $1/(\lambda_i-e_j^T(\mathcal{F}(C_1e_1)))$, $i=1,\ldots,2m(p+1)$, $j=1,\ldots,\ell$, since $L$ is diagonal, we can write
$$L^{-1}\text{vec}(S_m^{-1}E_1\pmb{\gamma}(\mathcal{F}[e_1,\tau F_2])^T)=\text{vec}\left(\mathcal{H}\odot\left(S_m^{-1}E_1\pmb{\gamma}(\mathcal{F}[e_1,\tau F_2])^T\right)\right),$$
so that 
$$N^TL^{-1}\text{vec}(S_m^{-1}E_1\pmb{\gamma}(\mathcal{F}[e_1,\tau F_2])^T)=\left(\mathcal{H}\odot\left(S_m^{-1}E_1\pmb{\gamma}(\mathcal{F}[e_1,\tau F_2])^T\right)\right)\mathcal{F}^{-T}e_\ell.
$$

We now have a closer look at the matrix $N^TL^{-1}M$ in \eqref{Y_m_circulant3}. The $(i,j)$-th entry of this matrix can be written as
\begin{align}\label{NLM}
e_i^TN^TL^{-1}Me_j&=e_i^T(\mathcal{F}^{-T}e_\ell\otimes I_{2m(p+1)})^TL^{-1}(\mathcal{F}e_1\otimes I_{2m(p+1)})e_j=\text{vec}(e_ie_\ell^T\mathcal{F}^{-1})^TL^{-1}\text{vec}(e_je_1^T\mathcal{F}^{T})\notag\\
&=\text{vec}(e_ie_\ell^T\mathcal{F}^{-1})^T\text{vec}\left(\mathcal{H}\odot\left(e_je_1^T\mathcal{F}^{T}\right)\right)\notag=\langle \mathcal{H}\odot\left(e_je_1^T\mathcal{F}^{T}\right),e_ie_\ell^T\mathcal{F}^{-1}\rangle_F\\
&=\text{trace}\left(\mathcal{F}^{-T}e_\ell e_i^T\left(
\mathcal{H}\odot\left(e_je_1^T\mathcal{F}^{T}\right)\right)\right)= e_i^T\left(
\mathcal{H}\odot\left(e_je_1^T\mathcal{F}^{T}\right)\right)\mathcal{F}^{-T}e_\ell.
\end{align}
Note the abuse of notation in the derivation above: $e_i,e_j$ denote the canonical basis vectors of $\mathbb{R}^{2m(p+1)}$ whereas $e_1,e_\ell$ the ones of $\mathbb{R}^{\ell}$.

An important property of the Hadamard product says that for any real vectors $x,y$ and matrices $A,B$ of conforming dimensions, we can write $x^T(A\odot B)y=\text{trace}(\text{diag}(x)A\text{diag}(y)B^T)$. By applying this result to~\eqref{NLM}, we get
\begin{align}\label{NLM2}
e_i^TN^TL^{-1}Me_j&=\text{trace}\left( \text{diag}(e_i)\mathcal{H}\text{diag}(\mathcal{F}^{-T}e_\ell)\mathcal{F}e_1e_j^T\right)=e_j^T\text{diag}(e_i)\mathcal{H}\left(\mathcal{F}^{-T}e_\ell\odot\mathcal{F}e_1\right)\notag\\
&=e_j^Te_ie_i^T\mathcal{H}\left(\mathcal{F}^{-T}e_\ell\odot\mathcal{F}e_1\right)=\delta_{i,j}e_i^T\mathcal{H}\left(\mathcal{F}^{-T}e_\ell\odot\mathcal{F}e_1\right),
\end{align}
where $\delta_{i,j}$ denotes the Kronecker delta, i.e., $\delta_{i,i}=1$ and $\delta_{i,j}=0$ otherwise. Equation \eqref{NLM2} says that 
$N^TL^{-1}M$ is a diagonal matrix such that 
$N^TL^{-1}M=\text{diag}\left(\mathcal{H}\left(\mathcal{F}^{-T}e_\ell\odot\mathcal{F}e_1\right)\right)$.

The vector $w:=M(I_{2m(p+1)}+N^TL^{-1}M)^{-1}N^TL^{-1}\text{vec}(S_m^{-1}E_1\pmb{\gamma}(\mathcal{F}[e_1,\tau F_2])^T)$ in \eqref{Y_m_circulant2} can thus be computed by performing
$$ w=\text{vec}\left(\left(\left(I_{2m(p+1)}+\text{diag}\left(\mathcal{H}\left(\mathcal{F}^{-T}e_\ell\odot\mathcal{F}e_1\right)\right)\right)^{-1}\left(\mathcal{H}\odot\left(S_m^{-1}E_1\pmb{\gamma}(\mathcal{F}[e_1,\tau F_2])^T\right)\right)\mathcal{F}^{-T}e_\ell\right)e_1^T\mathcal{F}^T\right).
$$
The linear solve $L^{-1}w$ can be still carried out by exploiting the Hadamard product and the matrix $\mathcal{H}$ as
{\small
$$L^{-1}w=\text{vec}\left(\mathcal{H}\odot\left(\left(\left(I_{2m(p+1)}+\text{diag}\left(\mathcal{H}\left(\mathcal{F}^{-T}e_\ell\odot\mathcal{F}e_1\right)\right)\right)^{-1}\left(\mathcal{H}\odot\left(S_m^{-1}E_1\pmb{\gamma}(\mathcal{F}[e_1,\tau F_2])^T\right)\right)\mathcal{F}^{-T}e_\ell\right)e_1^T\mathcal{F}^T\right)\right).$$
}
To conclude, the matrix $Y_m$ can be computed by
\begin{equation}\label{efficient_Ym}
 Y_m=S_m(Z-W)\mathcal{F}^{-T},\; \text{where }
 \begin{array}{l}
 Z=\mathcal{H}\odot\left(S_m^{-1}E_1\pmb{\gamma}(\mathcal{F}[e_1,\tau F_2])^T\right), \\ 
  W=\mathcal{H}\odot\left(\left(\left(I_{2m(p+1)}+\text{diag}\left(\mathcal{H}\left(\mathcal{F}^{-T}e_\ell\odot\mathcal{F}e_1\right)\right)\right)^{-1}Z\mathcal{F}^{-T}e_\ell\right)e_1^T\mathcal{F}^T\right),
 \end{array}
\end{equation}
and no Kronecker products are involved in such a computation.

The computation of $Y_m$ by \eqref{efficient_Ym} requires $\mathcal{O}\left(8m^3(p+1)^3+(\log\ell+4m^2(p+1)^2)\ell\right)$ floating point operations (flops) that has to be compared with the
$\mathcal{O}\left(8m^3(p+1)^3+4m^2(p+1)^2\ell\right)$ flops needed to calculate $Y_m$ by~\eqref{Y_comp}. Even though the presence of the FFT makes the asymptotic cost of \eqref{efficient_Ym} slightly larger than the one of \eqref{Y_comp}, performing \eqref{efficient_Ym} is usually much faster than \eqref{Y_comp} in terms of actual computational time. 
Indeed, no {\tt for} loops are required in \eqref{efficient_Ym} while efficient BLAS 3 operations can be exploited. Moreover, many of the computations involving the FFT can be performed once and for all at the beginning of the iterative process.

The discrete Fourier transform matrix $\mathcal{F}$ is never explicitly assembled and in all the experiments reported in section~\ref{Numerical results} its action
and the action of its inverse have been performed by means of the Matlab function {\tt fft} and {\tt ifft} respectively.
 
We would like to point out that the novel strategy presented in this section can be applied as a direct solver to equation~\eqref{Sylv_eq3} whenever the eigendecomposition of $\bigotimes_{i=1}^d(I_n-\mathcal{P}_1)+\tau\widebar K_d$ can be computed, e.g., if~\eqref{heat_eq} is discretized on a coarse spatial grid or if this matrix can be cheaply diagonalized by, e.g., sine transforms as considered in \cite{McDonald2018}.

\subsection{Multistep methods}\label{Multistep methods}
If a BDF of order $s$ is employed for the time discretization, with the same notation of section~\ref{A matrix equation formulation}, equation~\eqref{discrete_eq} has to be replaced by
\begin{equation}\label{discrete_eq_BDF}
 \frac{\mathbf{u}_k-\sum_{j=1}^s\alpha_j\mathbf{u}_{k-j}}{\tau\beta}+K_d\mathbf{u}_k=\mathbf{f}_k,
\end{equation}
where $\alpha_j=\alpha_j(s)$, $\beta=\beta(s)\in\mathbb{R}$ are the coefficients defining the selected BDF. See Table~\ref{BDF_coeff}\footnote{To have a consistent notation in the equations \eqref{Sylv_eq_BDF} and \eqref{Sylv_eq3}, we have changed sign to the $\alpha_j$'s with respect to the values listed in \cite[Table 5.3]{Ascher1998}.}.
It has been proved that for $s>6$ the BDFs become unstable, see, e.g., \cite[Section 5.2.3]{Ascher1998}, and we thus restrict ourselves to the case of $s\leq6$.

 \begin{table}[!ht]
 \centering
 \caption{Coefficients for the BDF of order $s$ for $s\leq 6$. See, e.g., \protect\cite[Table 5.3]{Ascher1998}. \label{BDF_coeff}}
\begin{tabular}{|r| r r r r r r r|}
 \hline
 $s$ & $\beta$ & $\alpha_1$ & $\alpha_2$ & $\alpha_3$ & $\alpha_4$ & $\alpha_5$ & $\alpha_6$  \\
 \hline
 \hline
 1 & 1 & 1 & &&&&\\
 2 & 2/3 & 4/3 & -1/3 &&&&\\
 3 & 6/11 & 18/11 & -9/11 & 2/11&&&\\
 4 & 12/25 & 48/25 & -36/25 &  16/25& -3/25&&\\
 5 & 60/137 & 300/137 & -300/137 & 200/137& -75/137& 12/137&\\
 6 & 60/147 & 360/147 & -450/147 & 400/147& -225/147& 72/147& -10/147\\
 \hline 
\end{tabular}
 \end{table}

 Following the discussion of section~\ref{A matrix equation formulation}, the discrete problem coming from an all-at-once approach for \eqref{discrete_eq_BDF} can be formulated in terms of the following Sylvester equation
 \begin{equation}\label{Sylv_eq_BDF}
  (I+\tau\beta K_d)\mathbf{U}-\mathbf{U}\left(\sum_{j=1}^s\alpha_j\Sigma_j^T\right)=\mathbf{u}_0e_1^T+\tau\beta[\mathbf{f}_1,\ldots,\mathbf{f}_\ell],
 \end{equation}
where $\Sigma_j$ denotes the $\ell\times\ell$ zero matrix having ones only in the $j$-th subdiagonal.
 
 We still assume that the right-hand side in \eqref{Sylv_eq_BDF} admits a low-rank representation. In particular,
 $[\mathbf{f}_1,\ldots,\mathbf{f}_\ell] =F_1F_2^T$. Noticing that the boundary conditions can be imposed as described in section~\ref{Imposing the boundary conditions} provided
$$
\widebar K_1=\begin{bmatrix}
             1/(\tau\beta) & & \\
               & \mathring{K}_1 &\\
                & & 1/(\tau\beta)\\                     
          \end{bmatrix},
$$
 the matrix equation we need to solve has the form 
\begin{equation}\label{Sylv_eq_BDF2}
 \left(\bigotimes_{i=1}^d(I_n-\mathcal{P}_1)+\tau\beta
\widebar K_d\right)\mathbf{U}-\mathbf{U}\left(\sum_{j=1}^s\alpha_j\Sigma_j^T\right)=[\mathbf{u}_0,F_1][e_1,\tau\beta F_2]^T,\quad d=1,2,3.
\end{equation}

The left projection for the space operator can be still carried out as illustrated in section~\ref{Left projection} and the employment of a BDF of order $s$, $1<s\leq 6$, only affects the inner problem formulation. Equation \eqref{projected_eq} must be replaced by 
\begin{equation}\label{projected_eq_BDF}
(\mathcal{I}_{m}+\tau\beta T_m)Y_m-Y_m\left(\sum_{j=1}^s\alpha_j\Sigma_j^T\right)=E_1\pmb{\gamma} [e_1,\tau\beta F_2]^T.
\end{equation}
Once $Y_m$ is computed, the residual norm
can be cheaply evaluated by
$$\|R_m\|_F=\tau\beta\|E_{m+1}^T\underline{T}_mY_m\|_F.
$$

As in the case of $s=1$, the solution of equation~\eqref{projected_eq_BDF} may be very expensive, especially for large $\ell$, and an efficient procedure for the calculation of $Y_m$ is thus necessary.
The solution scheme we are going to derive takes inspiration from the method discussed in section~\ref{Efficient inner solves}.
Indeed, we observe that 
\begin{equation}\label{circulant_BDF}
\sum_{j=1}^s\alpha_j\Sigma_j=C_s-[e_1,\ldots,e_s]\pmb{\alpha}_s[e_{\ell-s+1},\ldots,e_\ell]^T
 ,\quad 
  \pmb{\alpha}_s=\begin{bmatrix}
   \alpha_s & \cdots& &\cdots&\alpha_1 \\
   & \alpha_s &\cdots& \cdots& \alpha_2 \\
   & & \ddots &  & \vdots \\
   & & & \alpha_{s} & \alpha_{s-1}\\
   & & & & \alpha_s\\                                                                                                                 \end{bmatrix}\in\mathbb{R}^{s\times s},
\end{equation}
 where $C_s\in\mathbb{R}^{\ell\times\ell}$ is circulant and can be thus diagonalized by the FFT, namely $C_s=\mathcal{F}^{-1}\Pi_s\mathcal{F}$, $\Pi_s=\text{diag}(\mathcal{F}(C_se_1))$.
 Following section~\ref{Efficient inner solves}, we can write 
 {\small
 $$\text{vec}(\widetilde Y_m)=L^{-1}\text{vec}(S_m^{-1}E_1\pmb{\gamma}(\mathcal{F}[e_1,\tau\beta F_2])^T)-L^{-1}M(I_{2ms(p+1)}+N^TL^{-1}M)^{-1}N^TL^{-1}\text{vec}(S_m^{-1}E_1\pmb{\gamma}(\mathcal{F}[e_1,\tau\beta F_2])^T),$$
}
where now $L:=I_\ell\otimes\Lambda_m-\Pi_s\otimes I_{2m(p+1)}\in\mathbb{R}^{2m(p+1)\times 2m(p+1)}$ and $M:=\mathcal{F}[e_1,\ldots,e_s]\otimes I_{2m(p+1)},$ $N:=\mathcal{F}^{-T}[e_{\ell-s+1},\ldots,e_\ell]\pmb{\alpha}_s^T\otimes I_{2m(p+1)}\in\mathbb{R}^{2m(p+1)\ell\times 2ms(p+1)}$. As before, the action of $L^{-1}$ can be carried out by exploiting the matrix $\mathcal{H}$ and the Hadamard product. In particular,
$$L^{-1}\text{vec}(S_m^{-1}E_1\pmb{\gamma}(\mathcal{F}[e_1,\tau\beta F_2])^T)=\text{vec}\left(\mathcal{H}\odot 
\left(S_m^{-1}E_1\pmb{\gamma}(\mathcal{F}[e_1,\tau\beta F_2])^T\right)\right),$$
and
$$N^TL^{-1}\text{vec}(S_m^{-1}E_1\pmb{\gamma}(\mathcal{F}[e_1,\tau\beta F_2])^T)=\text{vec}\left(\left(\mathcal{H}\odot 
\left(S_m^{-1}E_1\pmb{\gamma}(\mathcal{F}[e_1,\tau\beta F_2])^T\right)\right)\mathcal{F}^{-T}[e_{\ell-s+1},\ldots,e_\ell]\pmb{\alpha}_s^T\right).$$

The inspection of the entries of the matrix $N^TL^{-1}M\in\mathbb{R}^{2ms(p+1)
\times 2ms(p+1)}$ is a bit more involved than before. With abuse of notation, we start by recalling that the vector $e_j\in\mathbb{R}^{2ms(p+1)}$, $j=1,\ldots, 2ms(p+1)$, can be written as 
$e_j=\text{vec}(e_ke_h^T)$, $e_k\in\mathbb{R}^{2m(p+1)}$, $e_h\in\mathbb{R}^{s}$, $j=k+2m(p+1)\cdot(h-1)$. Therefore,
\begin{align*}
 e_i^TN^TL^{-1}Me_j&=\text{vec}(e_re_q^T)^TN^TL^{-1}M\text{vec}(e_ke_h^T)\\
 &=\text{vec}(e_re_q^T\pmb{\alpha}_s[e_{\ell-s+1},\ldots,e_\ell]^T\mathcal{F}^{-1})^TL^{-1}\text{vec}(e_ke_h^T[e_1,\ldots,e_s]^T\mathcal{F}^T)\\
 &=\text{vec}(e_re_q^T\pmb{\alpha}_s[e_{\ell-s+1},\ldots,e_\ell]^T\mathcal{F}^{-1})^T\text{vec}\left(\mathcal{H}\odot \left(e_ke_h^T\mathcal{F}^T\right)\right)\\
 &=\left\langle \mathcal{H}\odot \left(e_ke_h^T\mathcal{F}^T\right), e_re_q^T\pmb{\alpha}_s[e_{\ell-s+1},\ldots,e_\ell]^T\mathcal{F}^{-1}\right\rangle_F \\
 &=\text{trace}\left(\mathcal{F}^{-T}
[e_{\ell-s+1},\ldots,e_\ell]\pmb{\alpha}_s^T
e_qe_r^T\left(\mathcal{H}\odot \left(e_ke_h^T\mathcal{F}^T\right)\right)\right)\\
&=e_r^T\left(\mathcal{H}\odot \left(e_ke_h^T\mathcal{F}^T\right)\right)
\mathcal{F}^{-T}
[e_{\ell-s+1},\ldots,e_\ell]\pmb{\alpha}_s^T
e_q.
\end{align*}
Notice that in the second step above we have $e_h^T[e_1,\ldots,e_s]^T=e_h^T$ and, differently from the one in the left-hand side where $e_h\in\mathbb{R}^s$, the vector in the right-hand side denotes the $h$-th canonical basis vector of $\mathbb{R}^\ell$, $h=1,\ldots,s$.

By exploiting the same property of the Hadamard product used in the derivation presented in section~\ref{Efficient inner solves}, we have
\begin{align}\label{SMW_BDF}
 e_i^TN^TL^{-1}Me_j&=\text{trace}\left(\text{diag}(e_r)\mathcal{H}\text{diag}(\mathcal{F}^{-T}
[e_{\ell-s+1},\ldots,e_\ell]\pmb{\alpha}_s^Te_q)\mathcal{F}e_he_k^T\right)\notag\\
&=e_k^T\text{diag}(e_r)\mathcal{H}\left(\mathcal{F}^{-T}
[e_{\ell-s+1},\ldots,e_\ell]\pmb{\alpha}_s^Te_q)\odot \mathcal{F}e_h\right)\notag\\
&=\delta_{k,r}e_r^T\mathcal{H}\left(\mathcal{F}^{-T}
[e_{\ell-s+1},\ldots,e_\ell]\pmb{\alpha}_s^Te_q)\odot \mathcal{F}e_h\right).
 \end{align}
Recalling that the indices in the above expression are such that $i=r+2m(p+1)\cdot(q-1)$ and $j=k+2m(p+1)\cdot(h-1)$, the relation in \eqref{SMW_BDF} means that $N^TL^{-1}M$ is a $s\times s$ block matrix with blocks of size $2m(p+1)$ which are all diagonal. The $(q,h)$-th block of $N^TL^{-1}M$ is given by $\text{diag}\left(\mathcal{H}\left(\mathcal{F}^{-T}
[e_{\ell-s+1},\ldots,e_\ell]\pmb{\alpha}_s^Te_q)\odot \mathcal{F}e_h\right)\right)$.

If $\mathcal{S}:=I+N^TL^{-1}M$ and $Z:=\mathcal{H}\odot 
\left(S_m^{-1}E_1\pmb{\gamma}(\mathcal{F}[e_1,\tau\beta F_2])^T\right)$,
then we denote by $P$ the $2m(p-1)\times s$ matrix such that 
$\text{vec}(P)=\mathcal{S}^{-1}\text{vec}\left(Z\mathcal{F}^{-T}[e_{\ell-s+1},\ldots,e_\ell]\pmb{\alpha}_s^T\right)$ and,
to conclude, the solution $Y_m$ of the reduced problems \eqref{projected_eq_BDF} can be computed by 
\begin{equation}\label{efficient_Ym_BDF}
 Y_m=S_m(Z-W)\mathcal{F}^{-T},\; \text{where }
 \begin{array}{l}
 Z=\mathcal{H}\odot\left(S_m^{-1}E_1\pmb{\gamma}(\mathcal{F}[e_1,\tau\beta  F_2])^T\right), \\ 
  W=\mathcal{H}\odot\left(P [e_1,\ldots,e_s]^T\mathcal{F}^T\right).
 \end{array}
\end{equation}

A generic BDF of order $s$, $s\leq 6$, requires $s-1$ additional initial values $\mathbf{u}_1,\ldots,\mathbf{u}_{s-1}$
together with $\mathbf{u}_0$. 
If these values are known, we have to simply change the right-hand side in \eqref{Sylv_eq_BDF2} and consider 
$$[\sum_{j=1}^s\alpha_j \mathbf{u}_{s-j},\sum_{j=1}^{s-1}\alpha_{j+1} \mathbf{u}_{s-j}, \ldots,\alpha_s \mathbf{u}_{s-1}, F_1][e_1,e_2,\ldots,e_s,\tau\beta F_2]^T,$$
in place of 
$[\mathbf{u}_0,F_1][e_1,\tau\beta F_2]^T$. Therefore, we need to construct the space 
$$\mathbf{EK}_m^\square(\widebar K_d,[\sum_{j=1}^s\alpha_j \mathbf{u}_{s-j},\sum_{j=1}^{s-1}\alpha_{j+1} \mathbf{u}_{s-j}, \ldots,\alpha_s \mathbf{u}_{s-1}, F_1]).$$
Except for the fact that now $2(p+s)$ basis vectors are added to the computed space at each iteration, the main steps of the solution method remain the same. See Example~\ref{Ex.1}.

If $\mathbf{u}_1,\ldots,\mathbf{u}_{s-1}$ are not given, they must be carefully approximated and such a computation must be $\mathcal{O}(\tau^s)$ accurate to maintain the full convergence order of the method. In standard implementation of BDFs, the $k$-th initial value $\mathbf{u}_k$, $k=1,\ldots,s-1$, is computed by a BDF of order $k$ with a time-step $\tau_k$, $\tau_k\leq\tau$. See, e.g., \cite[Section 5.1.3]{Ascher1998}. Allowing for a variable time-stepping is crucial for preserving the convergence order of the method.

The solution scheme presented in this paper is designed for a uniform time grid and it is not able to automatically handle a variable time-stepping.
Therefore, even though the solution process is illustrated for a generic BDF of order $s\leq 6$, in the experiments reported in section~\ref{Numerical results} we make use of the implicit Euler scheme for the time discretization when the additional initial values $\mathbf{u}_1,\ldots,\mathbf{u}_{s-1}$ are not provided.

The generalization of the proposed algorithm to the case of variable, and more in general, adaptive time-stepping will be the topic of future works.

\section{The rational Krylov subspace method}\label{The rational Krylov subspace method}

In section~\ref{The extended Krylov subspace method} we have considered only the extended Krylov subspace for the projection of the discrete space operator. However, the framework presented in section~\ref{Left projection} can be easily adapted to handle different approximation spaces as, e.g., the rational Krylov subspace~\eqref{def.rational}.

If we need to solve equation~\eqref{Sylv_eq3}, we can construct the rational Krylov subspace $\mathbf{K}_m^\square(\widebar K_d,[\mathbf{u}_0,F_1],\pmb{\xi})=\text{Range}(V_m)$, $V_m=[\mathcal{V}_1,\ldots,\mathcal{V}_m]\in\mathbb{R}^{n\times m(p+1)}$, $\pmb{\xi}=(\xi_2,\ldots,\xi_m)^T\in\mathbb{C}^{m-1}$, and perform a left projection as illustrated in section~\ref{Left projection}. 
Therefore, we still look for an approximate solution $U_m$ of the form $U_m=V_mY_m$ where $Y_m\in\mathbb{R}^{m(p+1)\times m(p+1)}$ is computed by imposing a Galerkin condition on the residual matrix $R_m:=(\bigotimes_{i=1}^d(I-\mathcal{P}_1)+\tau \widebar K_d)V_mY_m-V_mY_m\Sigma_1^T-[\mathbf{u}_0,F_1][e_1,\tau F_2]^T$, i.e., we impose $V_m^TR_m=0$. Once again, this orthogonality condition is equivalent to computing $Y_m$ as the solution of the projected equation 
$$(\mathcal{I}_m+\tau T_m)Y_m-Y_m\Sigma_1^T=E_1\pmb{\gamma}[e_1,\tau F_2]^T,$$
where, as before, $T_m=V_m^T\widebar K_dV_m$ and $\mathcal{I}_m=V_m^T\left(\bigotimes_{i=1}^d(I-\mathcal{P}_1)\right)V_m$.
Also when the rational Krylov subspace
is selected as approximation space we perform an explicit projection to obtain $T_m$ and $\mathcal{I}_m$ although, in exact arithmetic, the matrix $T_m$ can be computed by exploiting the results in \cite[Proposition 4.1]{Druskin2011}.   The solution $Y_m$ to the reduced equation can be still calculated by \eqref{efficient_Ym}.

Even though the main framework is similar to the one derived in section~\ref{The extended Krylov subspace method}, the employment of a rational Krylov subspace requires the careful implementation of certain technical aspects that we are going to discuss in the following.

The basis $V_m$ can be computed by an Arnoldi-like procedure as illustrated in \cite[Section 2]{Druskin2011} and it is well-known how the quality of the computed rational Krylov subspace deeply depends on the choice of the shifts $\pmb{\xi}$ employed in the basis construction. Effective shifts can be computed at the beginning of the iterative method if, e.g., 
some additional informations about the problem of interest are known. In practice, the shifts can be adaptively computed on the fly and the strategy presented in \cite{Druskin2011} can be employed to calculate the $(m+1)$-th shift $\xi_{m+1}$. The adaptive procedure proposed by Druskin and Simoncini in  
\cite{Druskin2011} only requires rough estimates of the smallest and largest eigenvalues of $\widebar K_d$ together with the Ritz values, i.e., the eigenvalues of the projected matrix $T_m$, that can be efficiently computed in $\mathcal{O}(m^3(p+1)^3)$ flops. In all the examples reported in section~\ref{Numerical results} such a scheme is adopted for the shifts computation.

For the rational Krylov subspace, the residual norm cannot be computed by performing \eqref{res_norm_comp} as an Arnoldi relation of the form 
$$\widebar K_dV_m=V_mT_m+\mathcal{V}_{m+1}E_{m+1}^T\underline{T}_m,$$ 
does not hold. 
An alternative but still cheap residual norm computation is derived in the next proposition.
\begin{Prop}\label{res_norm_comp_rational}
 At the $m$-th iteration of the rational Krylov subspace method, the residual matrix $R_m=(\bigotimes_{i=1}^d(I-\mathcal{P}_1)+\tau \widebar K_d)V_mY_m-V_mY_m\Sigma_1^T-[\mathbf{u}_0,F_1][e_1,\tau F_2]^T$ is such that
 $$\|R_m\|_F=\tau\left\| \left(\xi_{ m+1}I-(I-V_{m}V_{m}^T)\widebar K_d\right)\mathcal{V}_{m+1}E^T_{m+1}\underline{H}_mH_{m}^{-1} Y_m\right\|_F,
 $$
where the matrix $\underline{H}_{m}\in\mathbb{R}^{(m+1)\cdot(p+1)\times m(p+1)}$
collects the orthonormalization coefficients stemming from the ``rational'' Arnoldi procedure and $H_{m}\in\mathbb{R}^{m(p+1)\times  m(p+1)}$
is its principal square submatrix.
\end{Prop}
\begin{proof}
For the rational Krylov subspace
$\mathbf{K}_m^\square(\widebar K_d,[\mathbf{u}_0,F_1],\pmb{\xi})=\text{Range}(V_m)$, $V_m=[\mathcal{V}_1,\ldots,\mathcal{V}_m]$, the following Arnoldi-like relation holds
 \begin{equation}\label{rational_Arnoldi_rel}
\widebar K_dV_{m}
=
V_{m}T_{m}+\mathcal{V}_{m+1}E^T_{m+1}\underline{H}_{m}(\mbox{diag}(\xi_2,\ldots,\xi_{m+1})\otimes I_{p+1})H_{m}^{-1}
-(I-V_{m}V_{m}^T)\widebar K_d\mathcal{V}_{m+1}E^T_{m+1}\underline{H}_{m}H_{m}^{-1}.
 \end{equation}
See, e.g., \cite{Ruhe1994,Druskin2011}. 
Since the Arnoldi procedure is employed in the basis construction, $\underline{H}_m$ is a block upper Hessenberg matrix with block of size $p+1$ and we can write 
$$E^T_{m+1}\underline{H}_{m}(\mbox{diag}(\xi_2,\ldots,\xi_{m+1})\otimes I_{p+1})= \xi_{m+1} E^T_{m+1}\underline{H}_{m}.
$$
 The residual matrix $R_m$ is such that
 \begin{align*} 
    R_m  = & \left(\bigotimes_{i=1}^d(I-\mathcal{P}_1)+\tau \widebar K_d\right)V_mY_m-V_mY_m\Sigma_1^T-[\mathbf{u}_0,F_1][e_1,\tau F_2]^T\\
    =& V_m\left(\left(\mathcal{I}_m+\tau T_m\right)Y_m-Y_m\Sigma_1^T-E_1\pmb{\gamma}[e_1,\tau F_2]^T\right)\\
    &+\tau\left(\mathcal{V}_{m+1}E^T_{m+1}\underline{H}_{m}(\mbox{diag}(\xi_2,\ldots,\xi_{m+1})\otimes I_{p+1})H_{m}^{-1}
-(I-V_{m}V_{m}^T)\widebar K_d\mathcal{V}_{m+1}E^T_{m+1}\underline{H}_{m}H_{m}^{-1}\right)Y_m\\
    =&\tau\left(\xi_{m+1}\mathcal{V}_{m+1}E^T_{m+1}\underline{H}_{m}H_{m}^{-1}
-(I-V_{m}V_{m}^T)\widebar K_d\mathcal{V}_{m+1}E^T_{m+1}\underline{H}_{m}H_{m}^{-1}\right)Y_m,
    \end{align*}
and collecting the matrix $\mathcal{V}_{m+1}E^T_{m+1}\underline{H}_{m}H_{m}^{-1}$ we get the result.
\end{proof}

Proposition~\ref{res_norm_comp_rational} shows how the convergence check requires to compute the Frobenius norm of a $n\times m(p+1)$ matrix when the rational Krylov subspace is employed. This operation can be carried out in $\mathcal{O}(nm(p+1))$ flops by exploiting the cyclic property of the trace operator.

If $d=2,3$ and the initial values $u_0$, the source term $f$ and the boundary conditions $g$ are separable functions in the space variables, the same strategy presented in section~\ref{Structured space operators} can be adopted also when the rational Krylov subspace is selected in place of the extended one. We can compute $d$ rational Krylov subspaces corresponding to $d$ subspaces of $\mathbb{R}^{n}$ instead of one rational Krylov subspace contained in $\mathbb{R}^{n^d}$. Results similar to the one in Proposition~\ref{res_norm_comp_structured_op} can be derived by combining the arguments in the proof of 
Proposition~\ref{res_norm_comp_structured_op} with the Arnoldi-like relation~\eqref{rational_Arnoldi_rel}.

In this section we have assumed that the implicit Euler scheme is employed for the time integration. Some modifications are necessary to handle BDFs of higher order and the resulting scheme can be easily derived by following the discussion in section~\ref{Multistep methods}.
\section{The convection-diffusion equation}\label{The convection-diffusion equation}
In principle, the matrix reformulation presented in section~\ref{A matrix equation formulation}, and thus the solution process illustrated in section~\ref{The extended Krylov subspace method}-\ref{The rational Krylov subspace method}, can be applied to any PDEs of the form $u_t+\mathfrak{L}(u)=f$ where only space derivatives are involved in the linear differential operator  $\mathfrak{L}$.

In this section we provide some details in the case of the time-dependent convection-diffusion equation
\begin{equation}\label{condiff_eq}
 \begin{array}{rlll}
         u_t-\varepsilon\Delta u+\vec{w}\cdot\nabla u&=&f,& \quad \text{in }\Omega\times (0,T],\\
         u&=&g,& \quad \text{on } \partial\Omega,\\
         u(x,0)&=&u_0(x),&
        \end{array}
\end{equation}
where $\Omega\subset\mathbb{R}^d$ is regular, $\varepsilon>0$ is the viscosity parameter and the convection vector $\vec{w}=\vec{w}(x)$ is assumed to be incompressible, i.e., $\text{div}(\vec{w})=0$.

As already mentioned, if $K^{\text{cd}}_d\in\mathbb{R}^{n^d\times n^d}$ denotes the matrix stemming from the discretization of the convection-diffusion operator $\mathfrak{L}(u)=-\varepsilon\Delta u+\vec{w}\cdot\nabla u$ on $\widebar \Omega$, the same exact arguments of section~\ref{A matrix equation formulation} lead to the Sylvester matrix equation
$$(I_{n^d}+\tau K^{\text{cd}}_d)\mathbf{U}-\mathbf{U}\Sigma_1^T=[\mathbf{u}_0,F_1][e_1,\tau F_2],$$
when the backward Euler scheme is employed in the time integration.

If $d=1$ and $\vec{w}=\phi(x)$, the matrix $K_1^{\text{cd}}$ can be written as $K_1^{\text{cd}}=\varepsilon K_1+\Phi B_1$ where, as before, $K_1$ denotes the discrete negative laplacian whereas $B_1$ represents the discrete first derivative and the diagonal matrix $\Phi$ collects the nodal values $\phi(x_i)$ on its diagonal. 

In \cite{Palitta2016}, it has been shown that the 2- and 3D discrete convection-diffusion operators possess a Kronecker structure if the components of $\vec{w}$ are separable functions in the space variables.

If $\vec{w}=(\phi_1(x)\psi_1(y),\phi_2(x)\psi_2(y))$ and $\Phi_i$, $\Psi_i$ are diagonal matrices collecting on the diagonal the nodal values of the corresponding functions $\phi_i$, $\psi_i$, $i=1,2$, then
\begin{equation}\label{convdiff_2d}
K_2^{\text{cd}}=\varepsilon K_1\otimes I+\varepsilon I\otimes K_1+\Psi_1\otimes \Phi_1 B_1+\Psi_2B_1\otimes \Phi_2.
\end{equation}
See \cite[Proposition 1]{Palitta2016}. Analogously, if $d=3$ and $\vec{w}=(\phi_1(x)\psi_1(y)\upsilon_1(z),\phi_2(x)\psi_2(y)\upsilon_2(z),\phi_3(x)\psi_3(y)\upsilon_3(z))$, we can write
\begin{equation}\label{convdiff_3d}
K_3^{\text{cd}}=\varepsilon K_1\otimes I\otimes I+\varepsilon I\otimes K_1\otimes I+\varepsilon I\otimes I\otimes K_1+\Psi_1\otimes\Upsilon_1\otimes \Phi_1B_1+\Psi_2B_1\otimes\Upsilon_2\otimes \Phi_2+\Psi_3\otimes\Upsilon_3B_1\otimes \Phi_3.
\end{equation}
where, as before, the diagonal matrices $\Phi_i$, $\Psi_i$, $\Upsilon_i$ collect on the main diagonal the nodal values of the corresponding functions.
See \cite[Proposition 2]{Palitta2016}.

In this case, we can take advantage of the Kronecker structure of $K_d^{\text{cd}}$ to automatically include the boundary conditions in the matrix equation formulation of the time-dependent convection-diffusion equation. This can be done by combining the arguments of section~\ref{Imposing the boundary conditions} with the strategy presented in \cite[Section 3]{Palitta2016}.

Even though $K_d^{\text{cd}}$ still has a Kronecker structure, this cannot be 
exploited in general for reducing the cost of the basis generation for $d=2,3$ as it has been described in section~\ref{Structured space operators}, also when $u_0$, $f$ and $g$ are separable functions in the space variables. This is due to the presence of the extra terms containing $B_1$ in the definitions \eqref{convdiff_2d}-\eqref{convdiff_3d} of $K_d^{\text{cd}}$.
Indeed, $K_d^{\text{cd}}$ is no longer of the form $\sum_{i=1}^d I\otimes \cdots \otimes I\otimes A_i\otimes I\otimes\cdots\otimes I$ and the tensorized Krylov approach presented in \cite{Kressner2009} cannot be employed. This difficulty is strictly related to the fact that efficient projection methods for generic generalized Sylvester equations of the form
$$\sum_{j=1}^p A_iXB_i=C_1C_2^T,$$
have not been developed so far.
The available methods work well if the coefficient matrices $A_i$ and $B_i$ fulfill certain assumptions which may be difficult to meet in case of the discrete convection-diffusion operator. See, e.g, \cite{Jarlebring2018,Benner2013a,Powell2017,Shank2016} for more details about solvers for generalized matrix equations.

The matrix $K_d^{\text{cd}}$ can be expressed as $\sum_{i=1}^d I\otimes \cdots \otimes I\otimes A_i\otimes I\otimes\cdots\otimes I$ in some very particular cases. For instance, if $d=2$ and $\vec{w}=(\phi(x),\psi(y))$, then 
$$K_2^{\text{cd}}=(\varepsilon K_1+\Psi B_1)\otimes I+ I\otimes (\varepsilon K_1+ \Phi B_1).
$$
Therefore, if $\mathbf{u}_0=\pmb{\phi}_{u_0}\otimes\pmb{\psi}_{u_0}$ and $[\mathbf{f}_1,\ldots,\mathbf{f}_\ell]=F_1F_2^T =(\Phi_f\otimes\Psi_f)F_2^T,$ the spaces $\mathbf{EK}_m^\square(\varepsilon K_1+\Psi B_1,
[\pmb{\phi}_{u_0},\Phi_f])$
and $\mathbf{EK}_m^\square(\varepsilon K_1+\Phi B_1,
[\pmb{\psi}_{u_0},\Psi_f])$ can be constructed in place of $\mathbf{EK}_m^\square(K_2^{\text{cd}},
[\mathbf{u}_0,F_1])$. Similarly if the rational Krylov subspace is employed as approximation space.

\section{Numerical results}\label{Numerical results}
In this section we compare our new matrix equation approach with state-of-the-art procedures for the solution of the algebraic problem arising from the discretization of time-dependent PDEs.
Different solvers can be applied to \eqref{all_at_once_system} depending on how one interprets the underlying structure of the linear operator $\mathcal{A}$.
We reformulate \eqref{all_at_once_system} as a matrix equation but clearly $\mathcal{A}$ can be seen as a large structured matrix and well-known iterative techniques as, e.g., GMRES \cite{Saad1986}, can be employed in the solution of the linear system \eqref{all_at_once_system}.
The matrix $\mathcal{A}$ does not need to be explicitly assembled and its Kronecker structure can be exploited to perform ``matrix-vector'' products. Moreover, one should take advantage of the low-rank of the right-hand side $\text{vec}([\mathbf{u}_0,F_1][e_1,\tau F_2]^T)$ to reduce the memory consumption of the procedure. Indeed, if $n^d\ell$ is very large, we would like to avoid the allocation of any long $n^d\ell$ dimensional vectors and this can be done by rewriting the Krylov iteration in matrix form and equipping the Arnoldi procedure with a couple of low-rank truncations. These variants of Krylov schemes are usually referred to as low-rank Krylov methods and in the following we will apply low-rank GMRES (LR-GMRES) to the solution of 
\eqref{all_at_once_system}. 
See, e.g., \cite{Benner2013a,Hochbruck1995,Breiten2016,Stoll2015} for some low-rank Krylov procedures applied to the solution of linear matrix equations while \cite{Kuerschner2019} for details about how to preserve the convergence properties of the Krylov routines when low-rank truncations are performed. 


Both the aforementioned variants of GMRES needs to be preconditioned to achieve a fast convergence in terms of number of iterations. 
In \cite{McDonald2018}, it has been shown that the operator
$$
\begin{array}{rrll}
\mathfrak{P}:& \mathbb{R}^{n^d\ell}&\rightarrow& \mathbb{R}^{n^d\ell}\\
&x&\mapsto& (I_\ell\otimes(I_{n^d}+\tau K_d)-C_1\otimes I_{n^d})x,
\end{array}$$
is a good preconditioner for \eqref{all_at_once_system}. If right preconditioning is adopted, at each iteration of the selected Krylov procedure we have to solve an equation of the form $\mathfrak{P}\widehat v=v_m$, where $v_m$ denotes the last basis vector that has been computed. Again, many different procedures can be employed for this task.
In case of GMRES, we proceed as follows. We write
\begin{align*}
 \widehat v=&\mathfrak{P}^{-1} v_m=(I_\ell\otimes(I_{n^d}+\tau K_d)-C_1\otimes I_{n^d})^{-1} v_m\\
 =&(\mathcal{F}^{-1}\otimes I_{n^d})(I_\ell\otimes(I_{n^d}+\tau K_d)-\Pi_1\otimes I_{n^d})^{-1} (\mathcal{F}\otimes I_{n^d})v_m,
\end{align*}
and we solve the block diagonal linear system with 
$I_\ell\otimes(I_{n^d}+\tau K_d)-\Pi_1\otimes I_{n^d}$
by applying block-wise the algebraic multigrid method AGMG developed by Notay and coauthors \cite{Notay2010,Napov2012,Notay2012}. 

In the low-rank Krylov technique framework, the allocation of the full basis vector $v_m\in\mathbb{R}^{n^d\ell}$ is not allowed as we would lose all the benefits coming from the low-rank truncations. Since $\mathfrak{P}\widehat v=v_m$ can be recast in terms of a matrix equation, in case of LR-GMRES we can inexactly invert $\mathfrak{P}$ by applying few iterations of Algorithm~\ref{EK_algorithm}. Notice that in this case, due to the definition of $\mathfrak{P}$, the solution of the inner equations in Algorithm~\ref{EK_algorithm} is easier. Indeed, with the notation of section~\ref{Efficient inner solves}, we have $Y_m=S_mZ\mathcal{F}^{-T}$ at each iteration $m$. However, since the extra computational efforts of computing $Y_m$ by \eqref{efficient_Ym} turned out to be very moderate with respect to the cost of performing $Y_m=S_mZ\mathcal{F}^{-T}$, we decided to run few iterations\footnote{In all the reported examples we performed 10 iterations of Algorithm~\ref{EK_algorithm} at each outer iteration.} of Algorithm~\ref{EK_algorithm} with the original operator instead of the preconditioner $\mathfrak{P}$. 
This procedure can be seen as an inner-outer Krylov scheme \cite{Simoncini2002}.



The preconditioning techniques adopted within GMRES and LR-GMRES
are all nonlinear. We thus have to employ flexible variants of the outer Krylov routines, namely FGMRES \cite{Saad1993} and LR-FGMRES.

We would like to underline that  the concept of preconditioning does not really exist in the context of matrix equations. See, e.g., \cite[Section 4.4]{Simoncini2016}. The efficiency of our novel approach mainly relies on the effectiveness of the selected approximation space.

In the following we will denote our matrix equation approach by either EKSM, when the extended Krylov subspace is adopted, or
RKSM, if the rational Krylov subspace is employed as approximation space. 
The construction of both the extended Krylov subspace $\mathbf{EK}_m^\square(\widebar K_d,[\mathbf{u}_0,F_1])$ and the rational Krylov subspace $\mathbf{K}_m^\square(\widebar K_d,[\mathbf{u}_0,F_1],\pmb{\xi})$ requires the solution of linear systems with the coefficient matrix $\widebar K_d$ (or a shifted version of it). Except for Example~\ref{Ex.4}, these linear solves are carried out by means of the Matlab sparse direct solver \emph{backslash}. In particular, for EKSM, the LU factors of $\widebar K_d$ are computed once and for all at the beginning of the iterative procedure so that only triangular systems are solved during the basis construction. The time for such LU decomposition is always included in the reported results.

To sum up, we are going to compare EKSM and RKSM with FGMRES preconditioned by AGMG (FGMRES+AGMG) and LR-FGMRES preconditioned by EKSM (LR-FGMRES+EKSM). 
The performances of the different algorithms are compared in terms of both computational time and memory requirements. In particular, since all the methods we compare need to allocate the basis of a certain Krylov subspace, the storage demand of each algorithm consists in the dimension of the computed subspace. The memory requirements of the adopted schemes are summarized in Table~\ref{mem_requirements} where $m$ indicates the number of performed iterations.

{\renewcommand{\arraystretch}{1.5}%
 \begin{table}[!ht]
 \centering
 \caption{Storage demand of the compared methods. \label{mem_requirements}}
\begin{tabular}{c|c|c|c}
 EKSM & RKSM & FGMRES & LR-FGMRES\\
 \hline
 $2(m+1)(p+1)(n^d+\ell)$
 & $(m+1)(p+1)(n^d+\ell)$
 & $(2m+1)n^d\ell$ & $\displaystyle{(n^d+\ell)\left(\sum_{i=1}^m( r_i+z_i)+r_{m+1}\right)}$\\ 
 \end{tabular}
 \end{table}
}

For LR-FGMRES, $r_i$ and $z_i$ denote the rank of the low-rank matrix representing the $i$-th vector of the unpreconditioned and preconditioned basis respectively.

Notice that for separable problems where the strategy presented in section~\eqref{Structured space operators} can be applied, the memory requirements of EKSM and RKSM can be reduced to $2(m+1)\sum_{i=1}^dp_in+2^d(m+1)^d\prod_{i=1}^dp_i\ell$ and $(m+1)\sum_{i=1}^dp_in+(m+1)^d\prod_{i=1}^dp_i\ell$ respectively, where $p_i$ denotes the rank of the initial block used in the construction of the $i$-th Krylov subspace, $i=1,\ldots,d$.

If not stated otherwise, the tolerance of the final relative residual norm is always set to $10^{-6}$.

All results were obtained by running MATLAB R2017b  \cite{MATLAB} on a 
standard node of the Linux cluster Mechthild hosted at the Max Planck Institute for Dynamics of Complex Technical Systems in Magdeburg, Germany\footnote{See {\tt https://www.mpi-magdeburg.mpg.de/cluster/mechthild} for further details.}.

We would like to mention that the operator $\mathcal{A}$ in \eqref{all_at_once_system} can be seen also as a tensor. In this case, the algebraic problem stemming from the discretization scheme thus amount to a tensor equation for which different solvers have been proposed in the recent literature. See, e.g.,\cite{Dolgov2013,Ballani2013,Dolgov2014,Andreev2015}.
To the best of our knowledge, all the routines for tensor equations available in the literature include a rank truncation step to reduce the storage demand of the overall procedure. Most of the time, a user-specified, constant  rank $r$ is employed in such truncations and determining the  value of $r$ which provides the \emph{best} trade off between accuracy and memory reduction is a very tricky task while the performance of the adopted scheme deeply depends on this selection. See, e.g., \cite[Section 4]{Andreev2015}. This drawback does not affect our matrix equation schemes where no rank truncation is performed while moderate memory requirements are still achieved as illustrated in the following examples. Moreover, tensor techniques are specifically designed for solving high dimensional PDEs and we believe they are one of the few multilinear algebra tools that are able to deal with the peculiar issues of such problems. However, here we consider problems whose dimensionality is at most 4 ($d=3$ in space and one dimension in time). 
Due to the aspects outlined above, we refrain from comparing our matrix equation schemes with tensor approaches as a fair numerical comparison is difficult to perform.

\begin{num_example}\label{Ex.1}
{\rm
Before comparing EKSM and RKSM with other solvers we would like to show first how our novel reformulation of the algebraic problem in terms of a Sylvester matrix equation is able to maintain the convergence order of the adopted discretization schemes. In particular, we present only the results obtained by EKSM as the ones achieved by applying RKSM are very similar.

We consider the following 1D problem
\begin{equation}\label{Ex.1_eq}
 \begin{array}{rlll}
         u_t&=&\Delta u,& \quad \text{in }(0,\pi)\times (0,1],\\
         u(0)=u(\pi)&=&0,&\\
         u(x,0)&=&\sin(x).&
        \end{array}
\end{equation}
This is a toy problem as the exact solution is known in closed form and it is given by $u(x,t)=\sin(x)e^{-t}$. With $u$ at hand, we are able to calculate the discretization error provided by our solution process. 

Equation~\eqref{Ex.1_eq} is discretized by means of second order centered finite differences in space and a BDF of order $s$, $s\leq 6$, in time.

In the following we denote by $U_m\in\mathbb{R}^{n\times \ell}$ the 
approximate solution computed by EKSM, by $U$ the $n\times \ell$ matrix whose $i$-th column represents the exact solution evaluated on the space nodal values at time $t_i$ whereas $\mathbf{U}\in\mathbb{R}^{n\times \ell}$ collects the $\ell$ vectors computed by sequentially solving the linear systems in \eqref{discrete_eq} by backslash.

We first solve the algebraic problem by EKSM with a tolerance $\epsilon=10^{-10}$ and we compare the obtained $U_m$ with $\mathbf{U}$.
In Table~\ref{tab1.1} we report the results for $n=4096$, $s=1$ and different values of $\ell$.

\begin{table}[!ht]
 \centering
 \caption{Example \ref{Ex.1}. Results for different values of $\ell$. $n=4096$, $s=1$. \label{tab1.1}}
\begin{tabular}{|r| r r| r|r| }
\hline
 & \multicolumn{2}{c|}{EKSM}
 & backslash & \\
 $\ell$ & It. & Time (secs) & Time (secs) & $\|U_m-\mathbf{U}\|_F/\|\mathbf{U}\|_F$ \\
 \hline
1024&2 & 4.891e-2 & 5.697e-1 & 2.009e-10 \\

4096&2 & 6.094e-2 & 2.501e0 & 1.0066e-10 \\

16384&2 & 8.647e-2 & 9.912e0 & 9.931e-11 \\

65536&2 & 1.737e-1 & 3.964e1 & 1.069e-11 \\
\hline
\end{tabular}
 \end{table}

Looking at the timings reported in Table~\ref{tab1.1}, since EKSM requires two iterations to convergence for all the tested values of $\ell$, we can 
readily appreciate how the computational cost of our novel approach mildly depends on $\ell$ while the time for the sequential solution of the linear systems 
in \eqref{discrete_eq} linearly grows with the number of time steps.

Moreover, we see how, for this example, we can obtain a very small algebraic error $\|U_m-\mathbf{U}\|_F/\|\mathbf{U}\|_F$ by setting a strict tolerance on the relative residual norm computed by EKSM. This means that, when we compare $U_m$ with $U$, the discretization error is the quantity that contributes the most to $\|U_m-U\|_F/\|U\|_F$.
In Figure~\ref{fig:1} we plot $\|U_m-U\|_F/\|U\|_F$ for different values of $n$, $\ell$ and $s$. In particular, in the picture on the left we plot the relative error for $\ell=16384$ and $s=1$ while varying $n$. On the right, we fix $n=32768$ and we plot 
$\|U_m-U\|_F/\|U\|_F$ for different values of $\ell$ and $s=1,2,3$. Notice that by knowing the analytic expression of the solution $u$, for $s>1$ we are able to provide the $s-1$ additional initial conditions $\mathbf{u}_1,\ldots,\mathbf{u}_{s-1}$ and the extended Krylov subspace $\mathbf{EK}_m^\square(\widebar K_1,[\sum_{j=1}^s\alpha_j \mathbf{u}_{s-j},\sum_{j=1}^{s-1}\alpha_{j+1} \mathbf{u}_{s-j}, \ldots,\alpha_s \mathbf{u}_{s-1}])$ can be constructed as discussed in section~\ref{Multistep methods}.

From the plots in Figure~\ref{fig:1} we can recognize how the convergence order of the tested discretization schemes is always preserved. Similar results are obtained for larger values of $s$, namely $s=4,5,6$, provided either a larger $n$ or a space discretization scheme with a higher convergence order is employed.

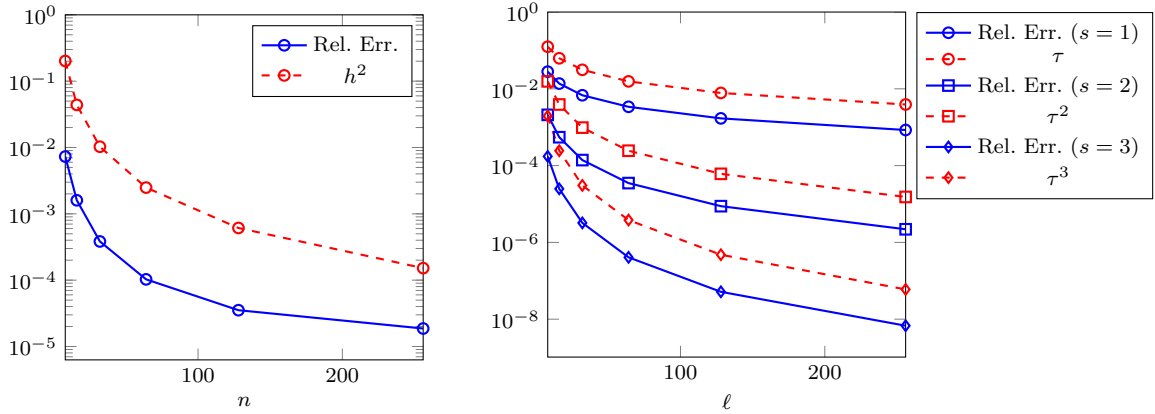
\begin{figure}
  \centering
  \caption{Example~\ref{Ex.1}. $\|U_m-U\|_F/\|U\|_F$ for different values of $n$, $\ell$ and $s$. Left: $\ell=16384$, $s=1$ while $n$ varies ($h$ denotes the space mesh size). Right: $n=32768$, $s=1,2,3$ while $\ell$ varies.} \label{fig:1}
  \begin{minipage}{.48\linewidth}
  	\centering 
  	\hspace{-2cm}
  \begin{tikzpicture}
    \begin{semilogyaxis}[width=0.8\linewidth, height=.27\textheight,
      legend pos =  north east,
      xlabel = $n$,xmin=8, xmax=256, ymax = 1e0]
      \addplot+[mark=o, color=blue, mark options={solid},thick] table[x index=0, y index=2]{data.dat};
       \addplot+ [mark=o, color=red, mark options={solid},thick,dashed]table[x index=0, y index=3]{data.dat};
        \legend{Rel. Err.,$h^2$};
    \end{semilogyaxis}          
  \end{tikzpicture}
  \end{minipage}~\begin{minipage}{0.48\textwidth}  	
    
  \centering
  	\hspace{-2cm}
\begin{tikzpicture}
    \begin{semilogyaxis}[width=0.8\linewidth, height=.27\textheight,
      legend pos = outer north east,
      xlabel = $\ell$,xmin=8, xmax=256, ymax = 1e0]
      \addplot+[mark=o, color=blue, mark options={solid},thick] table[x index=0, y index=4]{data_time.dat};
       \addplot+ [mark=o, color=red, mark options={solid},thick,dashed]table[x index=0, y index=1]{data_time.dat};
       
       \addplot+[mark=square, color=blue, mark options={solid},thick] table[x index=0, y index=5]{data_time.dat};
       \addplot+ [mark=square, color=red, mark options={solid},thick,dashed]table[x index=0, y index=2]{data_time.dat};
       
       \addplot+[mark=diamond, color=blue, mark options={solid},thick] table[x index=0, y index=6]{data_time.dat};
       \addplot+ [mark=diamond, color=red, mark options={solid},thick,dashed]table[x index=0, y index=3]{data_time.dat};
        \legend{Rel. Err. ($s=1$),$\tau$,Rel. Err. ($s=2$),$\tau^2$,Rel. Err. ($s=3$),$\tau^3$};
    \end{semilogyaxis}          
  \end{tikzpicture}

  \end{minipage}
 
\end{figure}

 }
\end{num_example}

\begin{num_example}\label{Ex.2}
{\rm
In the second example we consider 
the same equation presented in \cite[Section 6.1]{McDonald2018}.
This consists in the following 2D heat equation
\begin{equation}\label{Ex.2_eq}
 \begin{array}{rlll}
         u_t&=&\Delta u,&  \text{in }\Omega\times (0,1],\;\Omega:=(0,1)^2,\\
         u&=&0,&\text{on }\partial\Omega,\\
         u_0=u(x,y,0)&=&x(x-1)y(y-1).&
        \end{array}
\end{equation}
Equation~\eqref{Ex.2_eq} is discretized by means of second order centered finite differences in space and the backward Euler scheme in time.

Since the initial condition is a separable function in the space variables, and both the source term and the boundary conditions are zero, the strategy presented in section~\ref{Structured space operators} can be adopted. In particular if $\mathbf{u}_0$ denotes the $n^2$ vector collecting the values of $u_0$ for all the nodal values $(x_i,y_j)$, then we can write
$\mathbf{u}_0=\pmb{\phi}_{u_0}\otimes \pmb{\psi}_{u_0}$ where 
$\pmb{\phi}_{u_0}=[x_1(x_1-1),\ldots,x_n(x_n-1)]^T$, $\pmb{\psi}_{u_0}=[y_1(y_1-1),\ldots,y_n(y_n-1)]^T\in\mathbb{R}^n$. Therefore, the two extended Krylov subspaces $\mathbf{EK}_m^\square(\widebar K_1,\pmb{\phi}_{u_0})$ and $\mathbf{EK}_m^\square(\widebar K_1,\pmb{\psi}_{u_0})$ can be constructed in place of $\mathbf{EK}_m^\square(\widebar K_2,\mathbf{u}_0)$. Similarly for the rational Krylov subspace method.

In Table~\ref{tab2.1} we report the results for different values of $n$ and $\ell$.

\begin{table}[!ht]
 \centering
 \caption{Example \ref{Ex.2}. Results for different values of $n$ and $\ell$.\label{tab2.1}}
\begin{tabular}{|r r|r r|r r|rr|rr| }
\hline
 & & \multicolumn{2}{c|}{EKSM}
 & \multicolumn{2}{c|}{RKSM} & \multicolumn{2}{c|}{FGMRES+AGMG}& \multicolumn{2}{c|}{LR-FGMRES+EKSM}\\
 $n^2$ & $\ell$ & It. & Time (secs) & It. & Time (secs)&  It. & Time (secs)& It. & Time (secs)  \\
 \hline
4096 & 1024 & 6 & 2.487e-1 & 9& 3.313e-1 & 1 & 9.832e0 & 1& 1.899e-1 \\

& 4096 & 6 & 4.209e-1 & 9& 3.140e-1 &1 & 2.355e1 & 1& 1.747e-1 \\

& 16384 & 6 & 6.182e-1 & 9& 5.913e-1 & 1 & 7.025e1 & 1& 3.020e-1 \\

& 65536 & 6 & 1.671e0 & 9& 1.783e0 & 1 & 3.289e2 & 2& 4.001e0 \\

\hline
16384 & 1024 & 7 & 2.989e-1 & 11& 3.629e-1 & 1 & 3.662e1 & 2& 2.476e0 \\

& 4096 & 8 & 4.449e-1 & 11& 4.252e-1 & 1 & 1.135e2 & 2& 2.624e0 \\

& 16384 & 8 & 1.426e0 & 11& 1.089e0 & 1 & 3.418e2 & 2& 2.595e0 \\

& 65536 & 7 & 2.480e0 & 10& 2.349e0 & 1 & 1.483e3 & 2& 5.584e0 \\

\hline
65536 & 1024 & 8 & 4.071e-1 & 11& 3.887e-1 & 1 & 1.354e2 & 2& 1.992e1 \\

& 4096 & 10 & 9.726e-1 & 13& 5.540e-1 & 1 & 4.819e2 & 2& 1.980e1 \\

& 16384 & 10 & 1.916e0 & 13& 1.401e0 & 1 & 1.727e3 & 2& 2.141e1 \\

& 65536 & 10 & 5.469e0 & 11& 2.895e0 & OoM & OoM & 2& 1.654e1 \\

\hline

\end{tabular}
 \end{table}

As outlined in \cite{McDonald2018}, the preconditioner $\mathfrak{P}$ is very effective in reducing the total iteration count in FGMRES+AGMG and one FGMRES iteration is sufficient for reaching the desired accuracy for every value of $n$ and $\ell$ we tested. However, the preconditioning step is very costly in terms of computational time; this almost linearly grows with $\ell$. FGMRES+AGMG may benefit from the employment of a parallel implementation in the inversion of the block diagonal matrix $I_\ell\otimes \left(\bigotimes_{i=1}^2(I_{n}-\mathcal{P}_1)+\tau\widebar K_2\right)-\Pi_1\otimes I_{n^2}$.
Moreover, for the largest problem dimension we tested, the system returned an \emph{Out of Memory} (OoM) message as we are not able to allocate any $n^2\ell$ dimensional vectors.

LR-FGMRES+EKSM performs quite well in terms of computational time, especially for small $n$, and the number of iterations needed to converge is rather independent of both $n$ and $\ell$ confirming the quality of the inner-outer preconditioning technique.

Our new algorithms, EKSM and RKSM, are very fast. We would like to remind the reader that, for this example, the number of degrees of freedom (DoF) is equal to $n^2\ell$. This means that, for the finest refinement of the space and time grids we tested, our routines are able to solve a problem with $\mathcal{O}\left(4\cdot10^9\right)$ DoF in few seconds while reaching the desired accuracy. 

The number of iterations performed by EKSM and RKSM turns out to be very robust with respect to $\ell$ and the (almost) constant iteration count we obtain for a fixed $n$ lets us appreciate once more how the computational cost of our procedures modestly grows with $\ell$.

The robustness of our routines with respect to $\ell$ is not surprising. Indeed, the projection procedure we perform only involves the spatial component of the overall operator, namely $\bigotimes_{i=1}^2(I_n-\mathcal{P}_1)-\tau\widebar K_2$, and its effectiveness thus strictly depends on the
 spectral properties of $\bigotimes_{i=1}^2(I_n-\mathcal{P}_1)-\tau\widebar K_2$ which are mainly 
 fixed for a given $n$ although the mild dependence on $\ell$ due to the presence of the scalar $\tau$. 

Thanks to the separability of equation~\eqref{Ex.2_eq} and the employment of the strategy presented in section~\ref{Structured space operators}, EKSM and RKSM are very competitive also in terms of storage demand as illustrated in Table~\ref{tab2.2}.

{\renewcommand{\arraystretch}{1.2}%
\begin{table}[!ht]
 \centering
 \caption{Example \ref{Ex.2}. Memory requirements of the compared methods for different values of $n$ and $\ell$.\label{tab2.2}}
\begin{tabular}{|r r r r r r| }
\hline
 $n^2$ & $\ell$ & EKSM
 & RKSM & FGMRES+AGMG& LR-FGMRES+EKSM\\
 
 \hline
4096 & 1024 &28$n$+196$\ell$  &
20$n$+100$\ell$& 3$n^2\ell$& 18$(n^2+\ell)$\\
 & 4096 &28$n$+196$\ell$  &
20$n$+100$\ell$& 3$n^2\ell$& 18$(n^2+\ell)$\\
 & 16384 &28$n$+196$\ell$  &
20$n$+100$\ell$& 3$n^2\ell$& 18$(n^2+\ell)$\\
 & 65536 &28$n$+196$\ell$  &
20$n$+100$\ell$& 3$n^2\ell$& 80$(n^2+\ell)$\\

\hline

16384 & 1024 &32$n$+256$\ell$  &
24$n$+144$\ell$& 3$n^2\ell$& 84$(n^2+\ell)$\\
 & 4096 &36$n$+324$\ell$  &
24$n$+144$\ell$& 3$n^2\ell$& 84$(n^2+\ell)$\\
 & 16384 &36$n$+324$\ell$  &
24$n$+144$\ell$&3$n^2\ell$ & 85$(n^2+\ell)$\\
 & 65536 &32$n$+256$\ell$  &
22$n$+121$\ell$& 3$n^2\ell$& 86$(n^2+\ell)$\\
 \hline

65536 & 1024 &36$n$+324$\ell$  &
24$n$+144$\ell$& 3$n^2\ell$& 87$(n^2+\ell)$\\
 & 4096 &44$n$+484$\ell$  &
28$n$+196$\ell$& 3$n^2\ell$& 89$(n^2+\ell)$\\
 & 16384 &44$n$+484$\ell$  &
28$n$+196$\ell$& 3$n^2\ell$& 90$(n^2+\ell)$\\
 & 65536 &44$n$+484$\ell$  &
24$n$+144$\ell$& OoM& 90$(n^2+\ell)$\\
 \hline

\end{tabular}
 \end{table}
}
 }
\end{num_example}

\begin{num_example}\label{Ex.2.1}
{\rm
We now consider the isotropic diffusion example presented in \cite[Section 4.1]{Andreev2015} in the case of $d=3$. 
This problem consists in the following 3D heat equation
\begin{equation}\label{Ex.2.1_eq}
 \begin{array}{rlll}
         u_t&=&\Delta u+f,&  \text{in }\Omega\times (0,2],\;\Omega:=(-1,1)^3,\\
         u&=&0,&\text{on }\partial\Omega,\\
         u_0&=&0,&
        \end{array}
\end{equation}
where $f=f(x,y,z,t)=(1+\sin(\frac{\pi t}{2}))(1-x^2)e^x(1-y^2)e^y(1-z^2)e^z$.
Equation~\eqref{Ex.2.1_eq} is again discretized by means of second order centered finite differences in space and the backward Euler scheme in time.

Also for this example the strategy presented in section~\ref{Structured space operators} can be adopted. Indeed, both the initial condition and the boundary conditions are zero while the source term $f$ is a separable function in the space and time variables.  In particular, the discretization phase leads to a Sylvester equation of the form \eqref{Sylv_eq} where $\mathbf{u}_0=0$ and the matrix $[\mathbf{f}_1,\ldots,\mathbf{f}_\ell]$ can be written as 
$$[\mathbf{f}_1,\ldots,\mathbf{f}_\ell]=(\Phi_f\otimes\Psi_f\otimes\Upsilon_f)F_2^T,$$
where 
$\Phi_f=[(1-x_1^2)e^{x_1},\ldots,(1-x_n^2)e^{x_n}]^T$, $\Psi_f=[(1-y_1^2)e^{y_1},\ldots,(1-y_n^2)e^{y_n}]^T$,$\Upsilon_f=[(1-z_1^2)e^{z_1},\ldots,(1-z_n^2)e^{z_n}]^T\in\mathbb{R}^n$, and $F_2=[(1+\sin(\frac{\pi t_1}{2})),\ldots,(1+\sin(\frac{\pi t_\ell}{2}))]^T\in\mathbb{R}^{\ell}$. Therefore, the three extended Krylov subspaces $\mathbf{EK}_m^\square(\widebar K_1,\Phi_{f})$, $\mathbf{EK}_m^\square(\widebar K_1,\Psi_{f})$, and $\mathbf{EK}_m^\square(\widebar K_1,\Upsilon_{f})$ can be constructed in place of $\mathbf{EK}_m^\square(\widebar K_3,\Phi_f\otimes\Psi_f\otimes\Upsilon_f)$. Similarly for the rational Krylov subspace method.

In Table~\ref{tab2.1.1} we report the results for different values of $n$ and $\ell$. In particular, due to the very large number of DoFs adopted for this example, Table~\ref{tab2.1.1} depicts the performance of EKSM and RKSM only.

{\renewcommand{\arraystretch}{1.2}%
\begin{table}[!ht]
 \centering
 \caption{Example \ref{Ex.4}. Results for different values of $n$ and $\ell$.\label{tab2.1.1}}
\begin{tabular}{|r r |r r r|r r r| }
\hline
 & & \multicolumn{3}{c|}{EKSM}
 & \multicolumn{3}{c|}{RKSM} \\
   $n^3$ & $\ell$ & It. & Time (secs) & Mem. & It. & Time (secs)&  Mem.   \\
 \hline
 262144 & 1024 & 7 & 1.739e1  & $48n+4096\ell$ & 10  &2.139e0  & $33n+1331\ell$\\
 & 4096 & 7  & 1.918e0  & $48n+4096\ell$  & 11 & 2.196e0 &$36n+1728\ell$ \\
 & 16384 & 7 & 2.422e1 & $48n+4096\ell$ & 11 & 4.889e0 & $36n+1728\ell$\\
& 65536 & 7  & 4.157e1  & $48n+4096\ell$ & 11 & 1.647e1 & $36n+1728\ell$\\

 \hline

 2097152 & 1024 & 8  & 5.927e1  & $54n+5832\ell$ & 12 &6.994e0  & $39n+2197\ell$\\
 & 4096 & 8 & 5.993e1 & $54n+5832\ell$ & 11 & 4.718e0 & $36n+1728\ell$\\
 & 16384 & 9 & 1.787e2  & $60n+8000\ell$ & 12 & 7.105e0 & $39n+2197\ell$\\
& 65536 & 9 & 2.450e2 & $60n+8000\ell$ & 12 & 2.429e1 & $39n+2197\ell$\\

\hline

 16777216 & 1024 & 9 & 1.464e2 & $60n+8000\ell$ & 12 & 6819e0  & $39n+2197\ell$\\
 & 4096 & 9 & 1.502e2 & $60n+8000\ell$ & 14 & 2.618e1 & $45n+3375\ell$ \\
 & 16384 & 10 & 4.147e2 & $66n+10648\ell$ & 13 & 2.126e1 & $42n+2744\ell$\\
& 65536 & 10 & 5.644e2 & $66n+10648\ell$ & 13 & 3.560e1 & $42n+2744\ell$ \\

\hline

\end{tabular}
 \end{table}
}

We would like to stress one more time that even if \eqref{Ex.2.1_eq} amounts to a three-dimensional problem in space, the full exploitation of its separable structure leads to the employment of one-dimensional discrete operators in the basis construction. Therefore, the linear system solutions involved in both EKSM and RKSM can be efficiently performed by means of a sparse direct solver. 
Moreover, thanks to the strategy presented in section~\ref{Structured space operators}, EKSM and RKSM are very competitive also in terms of storage demand. For instance, for the finest refinement of the space and time grids we tested, which involves $\mathcal{O}(10^{12})$ DoFs, the whole RKSM procedure needs about the 0.015\% of the memory demand required by the allocation of the only right-hand side in the linear system formulation \eqref{eq.linear_system}.

We believe both EKSM and RKSM are very competitive also in terms of computational time as they manage to solve problems with a tremendous number of
DoFs in few seconds while always reaching the desired accuracy. From the results in Table~\ref{tab2.1.1} we can notice that the EKSM running time is always remarkably larger than the one achieved by RKSM, even though the number of iterations performed by the two routines is rather similar. This is due to the larger space constructed by EKSM and the consequent increment in the cost of the inner solutions. Indeed, at iteration $m$, 
the solution of equation \eqref{projected_eq} within
EKSM requires to compute the eigendecomposition of a $8m^3\times 8m^3$ matrix which costs $\mathcal{O}(512m^3)$ flops. On the other hand, a matrix of order $m^3$ is involved in the projected equation at the $m$-th RKSM
iteration so that the cost of its eigendecomposition is reduced to $\mathcal{O}(m^3)$ flops.

To conclude, also for this example the number of iterations performed by EKSM and RKSM turns out to be very robust with respect to $\ell$.

 }
\end{num_example}

\begin{num_example}\label{Ex.3}
{\rm
We consider another example coming from \cite{McDonald2018}. In particular, the problem we address is the following time-dependent convection-diffusion equation
\begin{equation}\label{Ex.3_eq}
 \begin{array}{rlll}
         u_t-\varepsilon\Delta u + \vec{w}\cdot\nabla u&=&0,&  \text{in }\Omega\times (0,1],\;\Omega:=(0,1)^2,\\
         u&=&g(x,y),&\text{on }\partial\Omega,\\
         u_0=u(x,y,0)&=&g(x,y)&\text{if }(x,y)\in\partial\Omega,\\
         u_0=u(x,y,0)&=&0&\text{otherwise,}         
        \end{array}
\end{equation}
where $\vec{w}=(2y(1-x^2),-2x(1-y^2))$ and $g(1,y)=g(x,0)=g(x,1)=0$ while $g(0,y)=1$.

This is a simple model for studying how the temperature in a cavity with a (constant) ``hot'' external wall ($\{0\}\times [0,1]$) distributes over time. The wind characterized by $\vec{w}$ determines a recirculating flow.

Once again, equation~\eqref{Ex.3_eq} is discretized by means of second order centered finite differences in space and the backward Euler scheme in time.

Thanks to the separability of $\vec{w}$, the spatial discrete operator $\widebar K_2^{\text{cd}}$ has a Kronecker structure and it can be written as in \eqref{convdiff_2d}. However, the presence of the extra terms containing the discrete first order derivative operator does not allow for the memory-saving strategy described in section~\ref{Structured space operators}. Nevertheless, the structure of $\widebar K_2^{\text{cd}}$ can be exploited to easily include the boundary conditions in the matrix equation formulation. Moreover, since the initial condition is equal to the boundary conditions on the boundary nodes and zero otherwise, 
the boundary conditions do not depend on time, and the source term is zero everywhere, the right-hand side of equation~\eqref{Sylv_eq3} can be written as $[\mathbf{u}_0,F_1][e_1,\tau [0,\mathbf{1}_{\ell-1}]^T]^T$ where, with a notation similar to the one used in section~\ref{Imposing the boundary conditions},
$F_1\in\mathbb{R}^{n^2}$ is such that $\mathcal{P}_2(\mathbf{u}_0e_1^T+\tau F_1 [0,\mathbf{1}_{\ell-1}]^T)=\mathcal{L}_2^{\text{cd}}\mathbf{U}$ on the boundary nodes and zero otherwise.
$\mathbf{1}_{\ell-1}\in\mathbb{R}^{\ell-1}$ denotes the vector of all ones.

Therefore, EKSM and RKSM construct the spaces $\mathbf{EK}_m^\square(\widebar K_2^{\text{cd}},[\mathbf{u}_0,F_1])$ and $\mathbf{K}_m^\square(\widebar K_2^{\text{cd}},[\mathbf{u}_0,F_1],\pmb{\xi})$ respectively.

In Table~\ref{tab3.1} we report the results for different values of $n$, $\ell$ and the viscosity parameter $\varepsilon$.

\begin{table}[!ht]
 \centering
 {\small
 \caption{Example \ref{Ex.3}. Results for different values of $n$, $\ell$ and $\varepsilon$.\label{tab3.1}}
\begin{tabular}{|r r r|r r|r r|rr|rr| }
\hline
 & & & \multicolumn{2}{c|}{EKSM}
 & \multicolumn{2}{c|}{RKSM} & \multicolumn{2}{c|}{FGMRES+AGMG}& \multicolumn{2}{c|}{LR-FGMRES+EKSM}\\
 $\varepsilon$ & $n^2$ & $\ell$ & It. & Time (secs) & It. & Time (secs)&  It. & Time (secs)& It. & Time (secs)  \\
 \hline
1& 4096 & 1024 & 13 & 3.977e-1 & 24& 1.186e0 &3 &3.465e1  & 3& 3.014e0 \\

&& 4096 & 14 & 3.459e-1 & 25& 1.322e0 & 3 & 8.375e1  & 2& 1.529e0 \\

&& 16384 & 14 & 8.421e-1 & 23& 1.613e0 & 3 & 2.624e2 & 2& 3.427e0 \\

&& 65536 & 13 & 2.333e0 & 24& 3.908e0 & 3 & 1.484e3 & 2 & 7.437e0 \\

 \hline
&16384 & 1024 & 15 & 2.072e0 & 26& 4.584e0 & 3 & 1.392e2 & 2& 5.995e0 \\

&& 4096 & 18 &2.947e0 & 27& 4.365e0 & 3  & 4.252e2 & 2& 6.847e0 \\

&& 16384 & 19 & 2.830e0 & 28& 5.571e0 &  3 & 1.309e3 &2 &8.445e0  \\

&& 65536 & 18 & 5.208e0 & 28& 7.472e0 & 3 & 6.732e3 &2 & 1.272e1\\

\hline
&65536 & 1024 & 17 & 1.720e1 & 32& 2.709e1 & 3  & 6.752e2 & 2& 3.669e1 \\

&& 4096 & 21 & 2.027e1 & 38& 3.338e1 & 3 & 1.967e3 &2 & 4.613e1 \\

&& 16384 & 24 & 2.426e1 &39 &3.507e1  & 3 & 6.616e3 & 3 & 1.187e2 \\

&& 65536 & 25 & 2.081e1 & 38& 3.552e1 & OoM & OoM & 3 & 1.330e2  \\

\hline
0.1& 4096 & 1024 & 15 & 5.347e-1 & 22& 1.277e0 & 4  & 2.094e1  & 2& 1.291e0 \\

&& 4096 & 14 & 4.380e-1 & 23& 1.174e0 &4  & 6.087e1 & 2& 1.367e0 \\

&& 16384 & 14 & 9.353e-1 & 23 & 1.678e0 & 4 & 2.679e2 &2 &2.662e0  \\

&& 65536 & 13 &2.447e0 & 20& 2.922e0 & 4 & 2.122e3 &2 & 6.083e0 \\

 \hline
&16384 & 1024 & 20 & 2.256e0 & 27& 4.769e0 & 4 & 1.118e2 & 2 & 5.465e0 \\

&& 4096 & 20 & 2.107e0 & 27& 4.605e0 &4  & 3.026e2 & 2& 5.009e0 \\

&& 16384 & 19 &2.977e0 & 24& 4.084e0 & 4 & 1.228e3 & 2 & 6.881e0 \\

&& 65536 & 19 & 5.593e0 & 26& 7.043e0 & 4 & 9.055e3 & 3& 5.314e1 \\

\hline
&65536 & 1024 & 25 & 2.261e1 & 35& 2.821e1 & 4 & 5.370e2 & 3&1.002e2  \\

&& 4096 & 27 & 1.607e1 & 32& 2.261e1 & 4 & 1.604e3 & 3&8.767e1  \\

&& 16384 & 26 & 1.623e1 & 31& 2.492e1 & 4 & 5.667e3 & 3 & 1.023e2  \\

&& 65536 & 25  & 2.062e1 & 30 & 2.417e1 & OoM & OoM & 3& 1.836e2 \\

\hline
0.01& 4096 & 1024 & 10 & 2.126e-1 & 16& 7.507e-1 & 8 & 2.751e1 & 2& 1.055e0 \\

&& 4096 & 9 & 2.509e-1 & 18& 9.415e-1 & 8 & 1.079e2 & 2& 9.823e-1 \\

&& 16384 & 9 & 4.855e-1 & 18& 1.235e0 & 6 & 4.339e2 & 2& 1.778e0 \\

&& 65536 & 10 &1.536e0 & 20& 2.467e0 & 6 & 3.932e3 & 2& 5.878e0 \\

 \hline
&16384 & 1024 & 13 & 1.333e0 & 18& 2.590e0 & 8 & 1.283e2& 2& 4.022e0 \\

&& 4096 & 12 & 1.304e0 & 20& 2.679e0 & 8 & 4.513e2 & 2& 3.841e0 \\

&& 16384 & 12 &1.579e0 & 22& 3.453e0 & 7 & 2.241e3 & 2& 4.876e0 \\

&& 65536 & 12 & 2.951e0 & 20& 4.575e0 & OoM & OoM & 2& 7.809e0 \\

\hline
&65536 & 1024 & 19 & 1.255e1 & 24& 1.508e1 & 9 & 7.083e2 & 2 & 2.823e1 \\

&& 4096 & 18 & 1.166e1 & 25& 1.727e1 & 7 & 1.763e3 & 2&  2.658e1\\

&& 16384 & 17 & 1.261e1 & 25& 1.815e1 & OoM & OoM &2 & 2.662e1 \\

&& 65536 & 17 & 1.393e1 & 22& 1.382e1 & OoM & OoM & 4& 1.448e2 \\

\hline

\end{tabular}
}
 \end{table}

 We can notice that the preconditioner $\mathfrak{P}$ within the FGMRES+AGMG procedure is still effective in reducing the outer iteration count.
 However, it seems its performance  depends on the viscosity parameter $\varepsilon$. Moreover, also for this example the preconditioning step leads to an overall computation time of FGMRES+AGMG that is not competitive when compared to the one achieved by the other solvers. As in Example~\ref{Ex.2}, an OoM message is returned whenever we try to allocate vectors of length $n^2\ell$ for $n^2=\ell=65536$. However, for this example, also for $n^2=16384$, $\ell=65536$, and $n^2=65536$, $\ell=16384$, with the viscosity parameter $\varepsilon=0.01$, the same error message is returned. Indeed, while the system is able to allocate only a moderate number of $n^2\ell$ dimensional vectors, FGMRES+AGMG needs a sizable number of iterations to converge so that the computed basis cannot be stored\footnote{In both cases, we are able to perform six FGMRES+AGMG iterations and the OoM message is returned while performing the seventh iteration. At the sixth iteration, the relative residual norm is $\mathcal{O}(10^{-6})$.}. A restarted procedure may alleviate such a shortcoming.

 LR-FGMRES+EKSM is very competitive in terms of running time as long as very few outer iterations are needed to converge. Indeed, its computational cost per iteration is not fixed but grows quite remarkably as the outer iterations proceed. This is mainly due to the preconditioning step. At each LR-FGMRES iteration $k$, EKSM is applied to an equation whose right-hand side is given by the low-rank matrix that represents the $k$-th basis vector of the computed space and the rank of such a matrix grows with $k$. This significantly increases the computational efforts needed to perform the 10 EKSM iterations prescribed as preconditioning step worsening the performance of the overall solution procedure.
 
 Also for this example, the new routines we propose in this paper perform quite well and the number of iterations mildly depends on $\ell$.
 
 The performances of our solvers are also pretty robust with respect to $\varepsilon$ and, especially for RKSM, it turns out that the number of iterations needed to converge gets smaller as the value of $\varepsilon$ is reduced. In the steady-state setting this phenomenon is well-understood. See, e.g., \cite[Section 4.2.2]{Elman2014}. In our framework, we can explain such a trend by adapting convergence results for RKSM applied to Lyapunov equations. Indeed, in \cite[Theorem 4.2]{Druskin2011a} it is shown how the convergence of RKSM for Lyapunov equations is guided by the maximum value of a certain rational function over the field of values $W(A):=\{z^*Az,\,z\in\mathbb{C}^n,\,\|z\|=1\}$ of the matrix $A$ used to define the employed rational Krylov subspace. Roughly speaking, the smaller $W(A)$, the better. In our context, even though we use $\widebar K_2^{\text{cd}}$ to build $\mathbf{K}_m^\square(\widebar K_2^{\text{cd}},[\mathbf{u}_0,F_1],\pmb{\xi})$, the projection technique involves the whole coefficient matrix $\bigotimes_{i=1}^2(I_n-\mathcal{P}_1)-\tau\overline K_2^{\text{cd}}$ and we thus believe it is reasonable to think that the success of RKSM relies on the field of values of such a matrix.
 In Figure~\ref{fig:3} we plot the field of values of $\bigotimes_{i=1}^2(I_n-\mathcal{P}_1)-\tau\overline K_2^{\text{cd}}$ for $n^2=65536$, $\ell=1024$, and different values of $\varepsilon$ and we can appreciate how such sets are nested and they get smaller when decreasing $\varepsilon$. This may intuitively explains the relation between the RKSM iteration count and $\varepsilon$ but further studies in this direction are necessary.

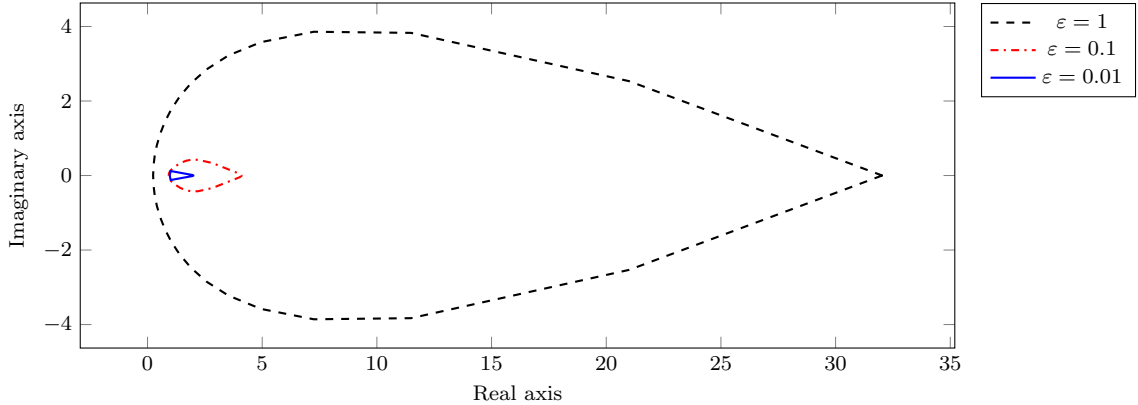
\begin{figure}
  \centering
  \caption{Example~\ref{Ex.3}. Field of values of $\bigotimes_{i=1}^2(I_n-\mathcal{P}_1)-\tau\overline K_2^{\text{cd}}$ for $n^2=65536$, $\ell=1024$ and different $\varepsilon$. } \label{fig:3}
  \begin{tikzpicture}
    \begin{axis}[width=0.8\linewidth, height=.27\textheight,
      legend pos =  outer north east,
      xlabel = Real axis, ylabel = Imaginary axis]
      \addplot+[mark=none, color=black, mark options={solid},thick,dashed] table[x index=0, y index=1]{fieldofvalues1BIS.dat};
      \addplot+[mark=none, color=red, mark options={solid},thick, dash dot] table[x index=0, y index=1]{fieldofvalues01BIS.dat};
      \addplot+[mark=none, color=blue, mark options={solid},thick] table[x index=0, y index=1]{fieldofvalues001BIS.dat};
       \legend{$\varepsilon=1$,$\varepsilon=0.1$,$\varepsilon=0.01$};
    \end{axis}          
  \end{tikzpicture}
\end{figure}
 
 Even though the approach presented in section~\ref{Structured space operators} cannot be adopted in this example, EKSM and RKSM are still very competitive also in terms of storage demand as illustrated in Table~\ref{tab3.2}.

 {\renewcommand{\arraystretch}{1.2}%
\begin{table}[!ht]
 \centering
 \caption{Example \ref{Ex.3}. Memory requirements of the compared methods for different values of $n$, $\ell$ and $\varepsilon$.\label{tab3.2}}
\begin{tabular}{|r r r r r r r| }
\hline
 $\varepsilon$ & $n^2$ & $\ell$ & EKSM
 & RKSM & FGMRES+AGMG& LR-FGMRES+EKSM\\
 
 \hline
1& 4096 & 1024 &56$(n^2+\ell)$  &
50$(n^2+\ell)$& 7$n^2\ell$& 659$(n^2+\ell)$\\

& & 4096 &60$(n^2+\ell)$  &
52$(n^2+\ell)$& 7$n^2\ell$& 324$(n^2+\ell)$\\

&& 16384 &60$(n^2+\ell)$  &
48$(n^2+\ell)$& 7$n^2\ell$& 323$(n^2+\ell)$\\
 
 && 65536 &56$(n^2+\ell)$  &
50$(n^2+\ell)$& 7$n^2\ell$& 234$(n^2+\ell)$\\

\hline

&16384 & 1024 &64$(n^2+\ell)$  &
54$(n^2+\ell)$& 7$n^2\ell$& 325$(n^2+\ell)$\\

&& 4096 &76$(n^2+\ell)$  &
56$(n^2+\ell)$& 7$n^2\ell$& 372$(n^2+\ell)$\\

&& 16384 &80$(n^2+\ell)$  &
58$(n^2+\ell)$&7$n^2\ell$ & 379$(n^2+\ell)$\\

&& 65536 &76$(n^2+\ell)$  &
58$(n^2+\ell)$& 7$n^2\ell$& 332$(n^2+\ell)$\\
 \hline

&65536 & 1024 &72$(n^2+\ell)$  &
66$(n^2+\ell)$& 7$n^2\ell$& 327$(n^2+\ell)$\\

& & 4096 &88$(n^2+\ell)$  &
78$(n^2+\ell)$& 7$n^2\ell$& 402$(n^2+\ell)$\\

&& 16384 &100$(n^2+\ell)$  &
80$(n^2+\ell)$& 7$n^2\ell$& 1102$(n^2+\ell)$\\

& & 65536 &104$(n^2+\ell)$  &
78$(n^2+\ell)$& OoM& 1293$(n^2+\ell)$\\
 \hline

 0.1& 4096 & 1024 &64$(n^2+\ell)$  &
46$(n^2+\ell)$& 9$n^2\ell$& 330$(n^2+\ell)$\\

& & 4096 &60$(n^2+\ell)$  &
48$(n^2+\ell)$& 9$n^2\ell$& 302$(n^2+\ell)$\\

&& 16384 &60$(n^2+\ell)$  &
48$(n^2+\ell)$& 9$n^2\ell$& 259$(n^2+\ell)$\\
 
 && 65536 &56$(n^2+\ell)$  &
42$(n^2+\ell)$& 9$n^2\ell$& 167$(n^2+\ell)$\\

\hline

&16384 & 1024 &84$(n^2+\ell)$  &
56$(n^2+\ell)$& 9$n^2\ell$& 381$(n^2+\ell)$\\

&& 4096 &84$(n^2+\ell)$  &
56$(n^2+\ell)$& 9$n^2\ell$& 362$(n^2+\ell)$\\

&& 16384 &80$(n^2+\ell)$  &
50$(n^2+\ell)$&9$n^2\ell$ & 356$(n^2+\ell)$\\

&& 65536 &80$(n^2+\ell)$  &
54$(n^2+\ell)$& 9$n^2\ell$& 1198$(n^2+\ell)$\\
 \hline

&65536 & 1024 &104$(n^2+\ell)$  &
72$(n^2+\ell)$& 9$n^2\ell$& 955$(n^2+\ell)$\\

& & 4096 &112$(n^2+\ell)$  &
68$(n^2+\ell)$& 9$n^2\ell$& 1108$(n^2+\ell)$\\

&& 16384 &108$(n^2+\ell)$  &
64$(n^2+\ell)$& 9$n^2\ell$& 1213$(n^2+\ell)$\\

& & 65536 &104$(n^2+\ell)$  &
62$(n^2+\ell)$& OoM& 1662$(n^2+\ell)$\\
 \hline

 0.01& 4096 & 1024 &44$(n^2+\ell)$  &
34$(n^2+\ell)$& 17$n^2\ell$& 275$(n^2+\ell)$\\

& & 4096 &40$(n^2+\ell)$  &
38$(n^2+\ell)$& 17$n^2\ell$& 228$(n^2+\ell)$\\

&& 16384 &40$(n^2+\ell)$  &
38$(n^2+\ell)$& 13$n^2\ell$& 160$(n^2+\ell)$\\
 
 && 65536 &44$(n^2+\ell)$  &
42$(n^2+\ell)$& 13$n^2\ell$& 161$(n^2+\ell)$\\

\hline

&16384 & 1024 &56$(n^2+\ell)$  &
38$(n^2+\ell)$& 17$n^2\ell$& 302$(n^2+\ell)$\\

&& 4096 &52$(n^2+\ell)$  &
42$(n^2+\ell)$& 17$n^2\ell$& 279$(n^2+\ell)$\\

&& 16384 &52$(n^2+\ell)$  &
46$(n^2+\ell)$&15$n^2\ell$ & 259$(n^2+\ell)$\\

&& 65536 &52$(n^2+\ell)$  &
42$(n^2+\ell)$& OoM& 168$(n^2+\ell)$\\
 \hline

&65536 & 1024 &80$(n^2+\ell)$  &
26$(n^2+\ell)$& 19$n^2\ell$& 361$(n^2+\ell)$\\

& & 4096 &76$(n^2+\ell)$  &
52$(n^2+\ell)$& 15$n^2\ell$& 334$(n^2+\ell)$\\

&& 16384 &72$(n^2+\ell)$  &
52$(n^2+\ell)$& OoM& 292$(n^2+\ell)$\\

& & 65536 &72$(n^2+\ell)$  &
46$(n^2+\ell)$& OoM& 1659$(n^2+\ell)$\\
 \hline

\end{tabular}
 \end{table}
}

We conclude this example by showing that our routines are also able to identify the physical properties of the continuous solution we want to approximate. In Figure~\ref{fig:2} we report the solution computed by EKSM for the case $n^2=65536$ and $\ell=1024$. In particular, we report the solution at different time steps $t_1$, $t_{\ell/2}$, $t_\ell$ (left to right) and for different values of $\varepsilon$ (top to bottom). We remind the reader that our solution represents the temperature distribution in a cavity with a constant, hot external wall. Looking at Figure~\ref{fig:2}, we can appreciate how the temperature distributes quite evenly in our domain for $\varepsilon=1$. The smaller $\varepsilon$, the more viscous the media our temperature spreads in. Therefore, the temperature is different from zero only in a very restricted area of our domain, close to the hot wall, for $\varepsilon=0.1,0.01$. Notice that for $\varepsilon=0.01$ and $t_1$, the part of the domain where the temperature is nonzero is so narrow that is difficult to appreciate with the resolution of Figure~\ref{fig:2}. For $\varepsilon=0.1,0.01$ we can also see how the temperature stops being evenly distributed as for $\varepsilon=1$ but follows the circulating flow defined by the convection vector $\vec{w}$.

 \begin{figure}[ht!]
  \centering
 \caption{Example~\ref{Ex.3}. Computed solution at different time steps (left to right: $t_1$, $t_{\ell/2}$, $t_\ell$) and related to different values of $\varepsilon$ (top to bottom: $\varepsilon=1$, $\varepsilon=0.1$, $\varepsilon=0.01$). $n^2=65536$, $\ell=1024$.} \label{fig:2}
  
  \includegraphics[scale=0.8]{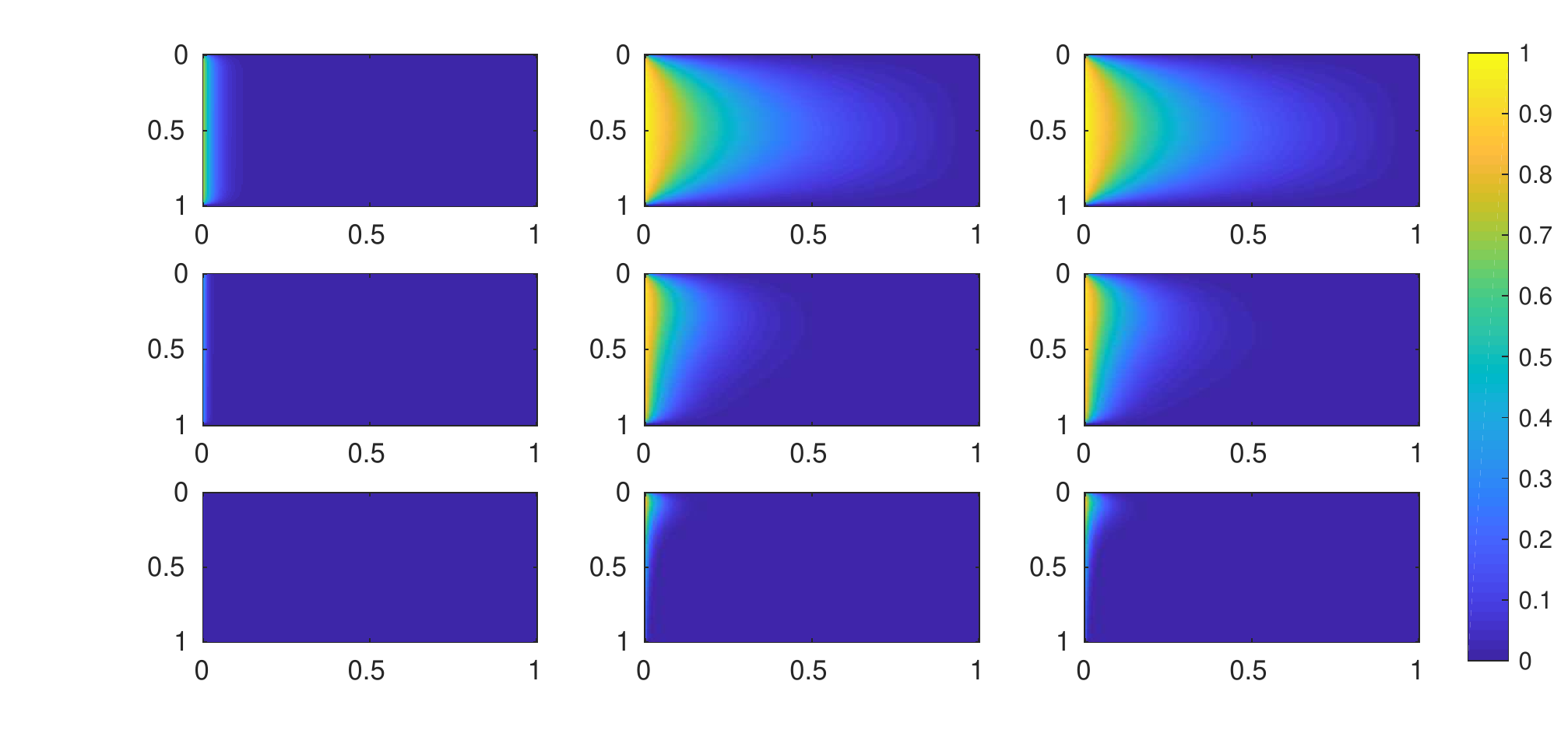}
  
\end{figure}

 }
\end{num_example}

\begin{num_example}\label{Ex.4}
{\rm
For the last example, we take inspiration from \cite[Example 5]{Palitta2016} and consider the following 3D time-dependent convection-diffusion equation
\begin{equation}\label{Ex.4_eq}
 \begin{array}{rlll}
         u_t-\Delta u + \vec{w}\cdot\nabla u&=&0,&  \text{in }\Omega\times (0,1],\;\Omega:=(0,1)^3,\\
         u&=&0,&\text{on }\partial\Omega,\\
         u_0&=&g,&\\
        \end{array}
\end{equation}
where $\vec{w}=(x\sin x,y\cos y, e^{z^2-1})$ and $g$ is such that 
\begin{equation}\label{equation_g}
 \begin{array}{rlll}
         -\Delta g + \vec{w}\cdot\nabla g&=&1,&  \text{in }\Omega,\\
         g&=&0,&\text{on }\partial\Omega.\\
        \end{array}
\end{equation}
Both \eqref{Ex.4_eq} and \eqref{equation_g} are discretized by centered finite differences in space and the backward Euler scheme is used for the time integration of \eqref{Ex.4_eq}. Once \eqref{equation_g} is discretized, we compute a numerical solution $\mathbf{g}\in\mathbb{R}^{n^3}$ by applying the strategy presented in, e.g., \cite{Palitta2016}, and then set $\mathbf{u}_0=\mathbf{g}$. 

Also in this example the convection vector $\vec{w}$ is a separable function in the space variables and the stiffness matrix $\widebar K_3^{\text{cd}}\in\mathbb{R}^{n^3\times n^3}$ can be written in terms of a Kronecker sum as illustrated in section~\ref{The convection-diffusion equation}. However, the initial value $\mathbf{u}_0$ is not separable in general and we have to employ $\mathbf{EK}_m^\square(\widebar K_3^{\text{cd}},\mathbf{u}_0)$ and $\mathbf{K}_m^\square(\widebar K_3^{\text{cd}},\mathbf{u}_0,\pmb{\xi})$
as approximation spaces.

It is well-known how 
sparse direct routines are not very well suited for solving linear systems with a coefficient matrix that stems from the discretization of a 3D differential operator, and
iterative methods perform better most of the time.
Therefore, the inner-outer GMRES method is employed 
to solve the linear systems involved
in the basis construction of both $\mathbf{EK}_m^\square(\widebar K_3^{\text{cd}},\mathbf{u}_0)$ and $\mathbf{K}_m^\square(\widebar K_3^{\text{cd}},\mathbf{u}_0,\pmb{\xi})$.
We set the tolerance on the relative residual norm for such linear systems equal to $10^{-8}$, i.e., two order of magnitude less than the outer tolerance. However, the novel results about inexact procedures in the basis construction of the rational and extended Krylov subspace presented in \cite{Kuerschner2018} may be adopted to  further reduce the computational cost of our schemes.

Due to the very large number $n^3\ell$ of DoFs we employ, in Table~\ref{tab4.1} we report only the results  for EKSM and RKSM.

{\renewcommand{\arraystretch}{1.2}%
\begin{table}[!ht]
 \centering
 \caption{Example \ref{Ex.4}. Results for different values of $n$ and $\ell$.\label{tab4.1}}
\begin{tabular}{|r r |r r r|r r r| }
\hline
 & & \multicolumn{3}{c|}{EKSM}
 & \multicolumn{3}{c|}{RKSM} \\
   $n^3$ & $\ell$ & It. & Time (secs) & Mem. & It. & Time (secs)&  Mem.   \\
 \hline
 32768 & 1024 & 10 & 1.026e1  & $22(n^3+\ell)$ & 12  &5.158e0  & $13(n^3+\ell)$\\
 & 4096 & 10  & 1.029e1  & $22(n^3+\ell)$  & 13 & 6.121e0 &$14(n^3+\ell)$ \\
 & 16384 & 10 & 1.705e1 & $22(n^3+\ell)$ & 13 & 5.479e0 & $14(n^3+\ell)$\\
& 65536 & 10  & 2.371e1  & $22(n^3+\ell)$ & 12 & 5.385e0 & $13(n^3+\ell)$\\

 \hline

 262144 & 1024 & 12  & 8.367e1  & $26(n^3+\ell)$ & 15 &4.378e1  & $16(n^3+\ell)$\\
 & 4096 & 13 & 9.287e1 & $28(n^3+\ell)$ & 16 & 4.326e1 & $17(n^3+\ell)$\\
 & 16384 & 13 & 9.109e1  & $28(n^3+\ell)$ & 15 & 4.296e1 & $16(n^3+\ell)$\\
& 65536 & 12 & 1.595e2 & $28(n^3+\ell)$ & 15 & 4.356e1 & $16(n^3+\ell)$\\

\hline

 2097152 & 1024 & 16 & 1.143e3 & $34(n^3+\ell)$ & 18 & 4.631e2  & $19(n^3+\ell)$\\
 & 4096 & 18 & 1.293e3 & $38(n^3+\ell)$ & 19 & 4.855e2 & $20(n^3+\ell)$ \\
 & 16384 & 18 & 1.298e3 & $38(n^3+\ell)$ & 18 & 4.541e2 & $19(n^3+\ell)$\\
& 65536 & 17 & 1.237e3 & $36(n^3+\ell)$ & 16 & 3.915e2 & $17(n^3+\ell)$ \\

\hline

\end{tabular}
 \end{table}
}
We can appreciate how our routines need a very reasonable time to meet the prescribed accuracy while maintaining a moderate storage consumption. For instance, the finest space and time grids we consider lead to a problem with $\mathcal{O}(10^{11})$ DoFs and RKSM manages to converge in few minutes by constructing a very low dimensional subspace.

It is interesting to notice how the computational time of RKSM is always much smaller than the one achieved by EKSM. This is due to the difference in the time devoted to the solution of the linear systems during the basis construction. Indeed, in RKSM, shifted linear systems of the form $\widebar K_3^{\text{cd}}-\xi_jI$ have to be solved and, in this example, it turns out that GMRES is able to achieve the prescribed accuracy in terms of relative residual norm in much fewer iterations than what it is able to do when solving linear systems with the only $\widebar K_3^{\text{cd}}$ as it is done in EKSM.

}
\end{num_example}

\section{Conclusions}\label{Conclusions}

In this paper we have shown how the discrete operator stemming from the discretization of time-dependent PDEs can be described in terms of a
matrix equation. For sake of simplicity, we have restricted our discussion to the heat equation and evolutionary convection-diffusion equations, but the same strategy can be applied to any PDE of the form $u_t+\mathfrak{L}(u)=f$ whenever $\mathfrak{L}(u)$ is a linear differential operator involving only space derivatives, provided certain assumptions on the source term $f$ and the boundary conditions are fulfilled. 

The matrix equation formulation of the discrete problem naturally encodes the separability of the spatial and time derivatives of the underlying differential operator. This lets us employ different strategies to deal with the spatial and time components of the algebraic problem and combine them in a very efficient solution procedure. In particular, state-of-the-art projection techniques have been proposed to tackle the spatial operator while the entry-wise structure of the time discrete operator has been exploited to derive effective solution schemes.

We have shown how to fully exploit the possible Kronecker structure of the stiffness matrix. Very good results are obtained also when this structure is not capitalized on in the solution process. 
This means that our approach can be successfully applied also to problems which do not lead to a stiffness matrix that possesses a Kronecker form as, e.g., in case of spatial domains $\Omega$ with a complex geometry or when sophisticated discretization methods (in space) are employed. We believe that also elaborate space-time adaptive techniques \cite{Deuflhard2012,Lang2001} can benefit from our novel approach. In particular, our routines can be employed to efficiently address the linear algebra phase within adaptive schemes for fixed time and space grids. Once the grids have been modified, our solvers can deal with the discrete operator defined on the newly generated time-space meshes. Both EKSM and RKSM can be easily implemented and we believe they can be incorporated in state-of-the-art software packages like, e.g., KARDOS \cite{ErdmannLangRoitzsch2002}.

As already mentioned, in the proposed approach the time step size $\tau$ is assumed to be fixed. We plan to extend our algorithm to the case of
adaptive time-stepping discretization schemes in the near future.

\section*{Acknowledgments}
We wish to thank Peter Benner, Jens Saak and Valeria Simoncini for insightful
comments on earlier versions of the manuscript. Their helpful suggestions are greatly appreciated. We also thank Jennifer Pestana for some observations on the preconditioning operator $\mathfrak{P}$. 

 The author is a member of the Italian INdAM Research group GNCS.

\bibliography{EvolPDEs}

\end{document}